\documentclass[twoside,11pt,preprint]{article}
\usepackage[utf8]{inputenc}
\usepackage{jmlr2e}
\usepackage[english]{babel}
\usepackage{mathtools}
\numberwithin{equation}{section}
\usepackage{xfrac}
\usepackage[matrix,arrow]{xy}
\usepackage{booktabs}
\usepackage{siunitx}

\usepackage{xcolor}

\newcommand{\bbF}{\mathbb{F}}
\newcommand{\bbN}{\mathbb{N}}
\newcommand{\bbR}{\mathbb{R}}
\newcommand{\calB}{\mathcal{B}}
\newcommand{\calC}{\mathcal{C}}
\newcommand{\calD}{\mathcal{D}}
\newcommand{\calF}{\mathcal{F}}
\newcommand{\calJ}{\mathcal{J}}
\newcommand{\calL}{\mathcal{L}}
\newcommand{\calO}{\mathcal{O}}
\newcommand{\calU}{\mathcal{U}}
\newcommand{\diff}{\mathrm{d}}
\newcommand{\ds}{\diff s}
\newcommand{\dt}{\diff t}
\newcommand{\dx}{\diff x}
\newcommand{\dd}[2][{}]{\frac{\!\diff#1}{\!\diff#2}}
\newcommand{\ddt}{\dd{t}}
\newcommand{\pp}[2][{}]{\frac{\partial#1}{\partial#2}}
\newcommand{\ppt}{\pp{t}}
\makeatletter
\newcommand{\argmin}{\mathop{\operator@font argmin}}
\newcommand{\Jac}{\mathop{\operator@font Jac}}
\newcommand{\pr}{\mathop{\operator@font pr}\nolimits}
\newcommand{\rang}{\mathop{\operator@font rang}}
\newcommand{\logloss}{\mathop{\operator@font H}\nolimits}
\makeatother

\newcommand{\To}{\longrightarrow}
\newcommand{\quand}{\quad\text{and}\quad}
\newcommand{\qquand}{\qquad\text{and}\qquad}

\newcommand{\powerp}{\mathsf{p}}

\setcitestyle{notesep={}}

\jmlrheading{23}{2022}{1-33}{11/21}{11/22}{campos00a}{%
  Cédric M. Campos, Alejandro Mahillo and David Martín de Diego%
}
\ShortHeadings{%
  Discrete Variational Calculus for Accelerated Optimization%
}{%
  C.M.~Campos, A.~Mahillo and D.~Martín de Diego%
}
\firstpageno{1}

\begin{document}
\title{%
  Discrete Variational Calculus for Accelerated Optimization%
}
\author{%
  \name Cédric M. Campos \email cedric.mcampos@urjc.es\\
  \addr Departamento de Matemática Aplicada, Ciencia e Ingeniería\\
  de los Materiales y Tecnología Electrónica\\
  Universidad Rey Juan Carlos\\
  Calle Tulipán s/n, 28933 Móstoles, Spain\\
  \AND
  \name Alejandro Mahillo \email almahill@unizar.es\\
  \addr Departamento de Matemáticas\\
  Instituto Universitario de Matemáticas y Aplicaciones\\
  Universidad de Zaragoza\\
  C. de Pedro Cerbuna, 12, 50009, Zaragoza, Spain\\
  \AND
  \name David Martín de Diego \email david.martin@icmat.es\\
  \addr Instituto de Ciencias Matemáticas (CSIC-UAM-UC3M-UCM)\\
  Calle Nicolás Cabrera 13-15, 28049 Madrid, Spain%
}

\editor{EDITOR(S) TO BE FILLED}

\maketitle

\begin{abstract}%
Many of the new developments in machine learning are connected with gradient-based optimization methods. Recently, these methods have been studied using a variational perspective \citep{BJW}. This has opened up the possibility of introducing variational and symplectic methods using geometric integration. In particular, in this paper, we introduce variational integrators \citep{marsden-west} which allow us to derive different methods for optimization. Using both, Hamilton's and Lagrange-d'Alembert's principle, we derive two families of optimization methods in one-to-one correspondence that generalize Polyak's heavy ball \citep{Po64} and Nesterov's accelerated gradient \citep{Nesterov}, the second of which mimics the behavior of the latter reducing the oscillations of classical momentum methods. However, since the systems considered are explicitly time-dependent, the preservation of symplecticity of autonomous systems occurs here solely on the fibers. Several experiments exemplify the result.
\end{abstract}


\section{Introduction}
Many of the literature on machine learning and data analysis is connected with gradient-based optimization methods \citep[see][; and references therein]{Polak-book,Book-Nesterov}. The computations often involve large data and parameter sets and then, not only the compuptational efficiency is a crucial point, but the optimization theory also plays a fundamental role. A typical optimization problem is:
\begin{equation}
  \label{eq:optim}
  \argmin f(x)\,,\quad x\in Q\,,
\end{equation}
where we assume that \(Q\) is a convex set in \(\bbR^n\) and \(f\) is a continuously differentiable convex function with Lipschitzian gradient. In this case one of the most extended algorithms for \eqref{eq:optim} is Nesterov’s accelerated gradient \citep{Nesterov,SuBoCa16} which may take the following form:
\begin{align*}
  y_{k+1} ={}& x_k - \eta\nabla f(x_k)\\
  x_{k+1} ={}& y_{k+1} + \frac{k}{k+3}(y_{k+1}-y_k)
\end{align*}
starting from an initial condition \(x_0\) (see more details in Sections \ref{sec:from-GD-to-NAG} and  \ref{sec:cm-nag}). An important observation was made by \citet{SuBoCa16} showing that the continuous limit of Nesterov's method is a time-dependent second order differential equation. Moreover, \citet{WiWiJo16} show that this system of differential equations has a variational origin \citep[see also][]{Wi16}. In particular, they take as point of departure this variational approach that captures acceleration in continuous time considering a particular type of time-dependent Lagrangian functions, called Bregman Lagrangians (see Section \ref{sec:bregman-lagrangians}).

In a recent paper, \citet{BJW} introduce symplectic (and presymplectic) integrators for the differential equations associated with accelerated optimizations methods \citep[see][, for an introduction to symplectic integration]{serna,hairer,blanes}. They use the Hamiltonian formalism since it is possible to extend the phase space to turn the system into a time-independent Hamiltonian system and apply there standard symplectic techniques \citep[see][]{MaOw16,CellEhO}. For recent improvements of this approach using adaptive Hamiltonian variational integrators, see \citet{DSLeok}.

In our paper we set an alternative route: The idea is to use variational integrators adapted to an explicit time-dependent framework and external forces \citep[see][, and references therein]{marsden-west} to derive a whole family of optimizations methods. The theory of discrete variational mechanics has reached maturity in recent years by combining results of differential geometry, classical mechanics and numerical integration. Roughly speaking, the continuous Lagrangian \(L\colon TQ\to\bbR\) is substituted by a discrete Lagrangian \(L_d\colon Q\times Q\to \bbR\). Observe that, by replacing the standard velocity phase space \(TQ\) with \(Q\times Q\), we are discretizing a velocity vector by two (in principle) close points. With the unique information of the discrete Lagrangian we can define the discrete action sum and, applying standard variational techniques, we derive a system of second order difference equations known as discrete Euler-Lagrange equations. The numerical order of the methods is obtained using variational error analysis \citep[see][]{marsden-west, PatrickCuell}. Moreover, it is possible to derive a discrete version of Noether's theorem relating the symmetries of the discrete Lagrangian with conserved quantities. The derived methods are automatically symplectic and, perhaps more importantly, easily adapted to other situations as, for instance, Lie group integrators, time-dependent Lagrangians, forced systems, optimal control theory, holonomic and nonholonomic mechanics, field theories, etc.

The Lagrangian functions depicted in Section \ref{sec:bregman-lagrangians}, Bregman Lagrangians, are those explicitly time-dependent that typically arise on accelerated optimization. The geometry for time-dependent systems is different from symplectic geometry, in particular, the phase space is odd dimensional. In this case, an appropriate geometric framework is given by cosymplectic geometry \citep[see][; and references therein]{Libermann,Nicola}. In Section \ref{sec:cosym} we introduce the cosymplectic structure associated to a time-dependent Hamiltonian system (induced by a time-dependent Lagrangian) and also an interesting symplectic preservation property associated to the restriction of the Hamiltonian flow to the fibres of the projection onto the time variable (Theorem \ref{thm:wer}). Having in mind this geometrical framework we introduce in Section \ref{sec:del} discrete variational mechanics for time-dependent Lagrangians with fixed time step \citep[compare with][, for variable time step]{marsden-west}. Moreover, we recover the symplectic character on fibres of the continuous Hamiltonian flow. We show the possibility to construct variational integrators using similar techniques to the developed for the autonomous case that, in some interesting cases, are in addition explicit and, consequently, reduce the computational cost. An example of such methods is the second-order difference equation
\[
x_{k+1}=x_k-\eta_k \nabla f(x_k)+\mu_k (x_k-x_{k-1}) \,,
\]
a type of momentum-descent method widely studied in the literature and whose origin goes back to \cite{Po64}.
Momentum methods allow to accelerate gradient descent by taking into account the ``speed'' achieved by the method at the last update. However, because of that speed, momentum methods can overpass the minimum. Nesterov's method tries to anticipate future information reducing the typical oscillations of classical momentum methods towards the minimum. In Section \ref{sec:forces} we adapt our construction of variational integrators to add external forces using discrete Lagrange-d'Alembert's principle \citep[see][]{marsden-west}. Upon this machinery, we derive in Section \ref{sec:cm-nag} two families of momentum methods in mutual bijective correspondence one of which corresponds to Nesterov's method (see Theorem \ref{thm:cm-vs-nag}).
Finally, for Section \ref{sec:simulations}, many methods and numerical simulations have been implemented in Julia v1.8.2. We optimize several test functions with our methodology and other methods that appeared recently in the literature. One of the test functions is reused afterwards for a machine learning example.

\section{From Gradient Descent to Nesterov's Accelerated Gradient}
\label{sec:from-GD-to-NAG}

In this section we give a historical perspective of Nesterov's accelerated gradient from gradient descent with a threefold objective: First, properly introduce the methods of interest and their properties, second, give an overall view of the elements to take under consideration, and, third, set some of the notation.

Although the first method that comes to mind to solve the optimization problem \eqref{eq:optim} is Newton-Raphson, the first ``dynamical'' one is due to
\citet{CauchyGD}. His method, known as Gradient Descent (GD), is the one-step method
\begin{equation}
  \label{eq:gd}
  x_{k+1}=x_k-\eta\nabla f(x_k) \,,
\end{equation}
where \(\eta\) is the step size parameter or, as it is referred in the machine learning community, the learning rate. It is readily seen that this method is a simple discretization of the first order ODE
\begin{equation}
  \label{eq:gd:ode}
  \dot x = -\nabla f(x) \,,
\end{equation}
from which it takes its dynamical nature. What is perhaps not so readily seen is that, given an initial condition \(x_0\), the trajectories obtained from both equations, \(x_k\) and \(x(t)\), converge to the argument of minima \(x^*\). In particular, \(x_k\) converges linearly to \(x^*\) while the function values \(f(x_k)\) do so to the global minimum \(f(x^*)\) at a rate of \(\calO(\sfrac1k)\) \citep{Po64,Po87}.

An initial improvement over GD was given by \citet{Po64}: He introduced
a novel two-step method, Polyak's Heavy Ball (PHB), also known as Classical Momentum (CM) after \cite{pmlr-v28-sutskever13}. As it was originally presented, PHB/CM takes the form of the two-step method
\begin{equation}
  \label{eq:cm:single}
  x_{k+1}=x_k-\eta P(x_k)+\mu(x_k-x_{k-1})
\end{equation}
where \(P\) is a functional operator for which a root is sought and \(\mu,\eta\) are ``small'' positive constants that condition the convergence of the method. In comparison with \eqref{eq:gd}, \eqref{eq:cm:single} adds a new term, \(x_k-x_{k-1}\), the momentum of the discrete motion that incorporates past information in an amount controlled by \(\mu\), the so called momentum coefficient. When \(P\) is conservative, that is, when \(P=\nabla f\), Polyak showed that, although the method's trajectory still converges linearly as GD's, it does so faster than GD's, that is, with a smaller geometric ratio \citep{Po64,Po87}.

The continuous analogue of \eqref{eq:cm:single} is the second order ODE
\begin{equation}
  \label{eq:cm:ode}
  \ddot x + \nu(t)\dot x + \eta(t)P(x) = 0
\end{equation}
that turns out to be the equation of motion of a Lagrangian system when \(P=\nabla f\) (Lemma \ref{thm:sodevar}). Then \(x(t)\) traces the motion of a point mass in a well given by \(f\). We therefore drop \(P\) and stick here on with \(\nabla f\).

A further an crucial step towards improving GD (and PHB/CM) was given in \citeyear{Nesterov} by \citeauthor{Nesterov}, a former student of Polyak. He presented a new method, coined after him as Nesterov's Accelerated Gradient (NAG), similar to PHB/CM but with a slight change of unexpected consequences. A naive derivation from \eqref{eq:cm:single} is almost immediate: Introduce a new variable \(y_k\) in \eqref{eq:cm:single} so it can be easily rewritten as the equivalent method
\begin{subequations}
  \label{eq:cm}
\begin{align}
  \label{eq:cm:tail}
  y_{k+1} ={}& x_k - \eta_k\nabla f(x_k)\\
  \label{eq:cm:head}
  x_{k+1} ={}& y_{k+1} + \mu_k(x_k-x_{k-1})
\end{align}
\end{subequations}
where discrete-time dependence has been added to the coefficients \(\mu,\eta\) for convenience. Replace the \(x\)'s of the momentum term (right hand side of the second equation \eqref{eq:cm:head}) by \(y\)'s to get the new and non-equivalent method
\begin{subequations}
  \label{eq:nag}
\begin{align}
  \label{eq:nag:tail}
  \bar y_{k+1} ={}& \bar x_k - \eta_k\nabla f(\bar x_k)\\
  \label{eq:nag:head}
  \bar x_{k+1} ={}& \bar y_{k+1} + \mu_k(\bar y_{k+1}-\bar y_k)
\end{align}
\end{subequations}
where the bars are added to distinguish more easily both methods and underline
that the sequences of points that they define are in fact different. This latter method \eqref{eq:nag} is NAG as it is usually presented. Nesterov showed in turn that its method accelerates the convergence rate of the function values down to \(\calO(\sfrac1{k^2})\) \citep[see][]{Nesterov,Book-Nesterov}.

The original values of \(\eta_k,\mu_k\) given by Nesterov are rather intricate, a simpler and commonly used version is
\begin{subequations}
  \label{eq:nag:magic}
\begin{align}
  \label{eq:nag:magic:tail}
  \bar y_{k+1} ={}& \bar x_k - \eta\nabla f(\bar x_k)\\
  \label{eq:nag:magic:head}
  \bar x_{k+1} ={}& \bar y_{k+1} + \frac{k}{k+3}(\bar y_{k+1}-\bar y_k)
\end{align}
\end{subequations}
with \(\eta>0\) constant. As it is shown in \citet{SuBoCa16}, a continuous analogue of \eqref{eq:nag:magic} is
\begin{equation}
  \label{eq:nag:magic:ode}
  \ddot x + \frac3t\dot x + \nabla f(x) = 0 \,,
\end{equation}
which is but a particular case of PHB/CM's continuous analogue \eqref{eq:cm:ode}. Besides \citet{SuBoCa16} also show that the function values converge to the minimum at an inverse quadratic rate, that is, \(f(x(t)) = f(x^*) + \calO(\sfrac1{t^2})\).

More generally (Remark \ref{rmk:force.evanescence}), \eqref{eq:nag} is a natural discretization of a perturbed ODE of the form
\begin{equation}
  \label{eq:cm:ode:perturbed}
  \ddot x + \nu(t)\dot x + \eta(t)\nabla f(x) = \varepsilon F(x,\dot x,t) \,,
\end{equation}
which also is the equation of motion of a Lagrangian system (Lemma \ref{thm:sodevar}). In fact, it is this variational origin that \citet{WiWiJo16} take as point of departure. Once a particular type of time-dependent Lagrangian functions is considered, a subfamily of the so called Bregman Lagrangians, the variational approach captures acceleration in continuous-time into the derived discrete schemes achieving, in this case, a function value convergence rate of \(\calO(t^{1-n})\) with \(n\geq3\) \citep[see also][]{Wi16}.

\section{Bregman Lagrangians}
\label{sec:bregman-lagrangians}
A Bregman Lagrangian is roughly speaking a time-dependent mechanical Lagrangian whose kinetic part is close to be a metric. They are built upon Bregman divergences \citep{bregman}, a particular case of divergence functions. Bregman Lagrangians allow to define variational problems whose solutions  minimize an objective function at an exponential rate \citep{BJW}.

A \textbf{divergence function} over a manifold \(Q\) is a twice differentiable function \(\calB\colon Q\times Q\to \bbR_+\) such that for all \(x, y\in Q\) we have:
\begin{itemize}
\item \(\calB(x,y)\geq 0\) and \(\calB(x,x)=0\);
\item \(\partial_x\calB(x,x)=\partial_y\calB(x,x)\); and
\item \(\partial^2_{xy}\calB(x,x)\) is negative-definite.
\end{itemize}

Divergence functions appear as pseudo-distances that are non-negative but are not, in general, symmetric. A typical divergence function over \(Q=\bbR^n\) associated to a differentiable strictly convex function \(\Phi\colon \bbR^n\to \bbR\) is the \textbf{Bregman divergence}:
\[
\calB_{\Phi}(x,y)=\Phi(x)-\Phi(y)-\langle d \Phi(y), x-y \rangle \,.
\]
Observe that it is the remainder of the first order Taylor expansion of \(\Phi\) around \(y\) evaluated at \(x\), a sort of Hessian metric.

Given a Bregman divergence over \(\bbR^n\), let us consider the time-dependent kinetic energy
\[
K(x,\dot{x},t)=\calB_{\Phi}(x+e^{-\alpha(t)}\dot{x}, x)
\]
and the time-dependent potential energy
\[
U(x,t)=e^{\beta(t)}f(x) \,,
\]
from which define the \textbf{Bregman Lagrangian} \(L\colon T\bbR^n\times T\bbR\to\bbR\) by
\begin{align*}
	L(x,\dot{x},t)={}&e^{\alpha(t)+\gamma(t)}\big(K(x,\dot{x},t)-U(x,t)\big)\\
	={}&e^{\alpha(t)+\gamma(t)}\left(\Phi(x+e^{-\alpha (t)}\dot{x})-\Phi(x) -e^{-\alpha (t)}\langle d\Phi(x), \dot{x}\rangle-e^{\beta(t)}f(x)\right) \,,
\end{align*}
where the time-dependent functions \(\alpha(t), \beta(t), \gamma(t)\) are chosen to produce different algorithms. These functions verify what \citet{WiWiJo16} refer to as \textbf{ideal scaling conditions}, namely,
\begin{equation}
  \label{eq:ideal-conds}
  \dot\gamma(t)=e^{\alpha(t)}
  \qquand
  \dot\beta(t)\leq e^{\alpha(t)}
  \,.
\end{equation}
The first condition greatly simplifies several expressions that can be derived from the Bregman Lagrangian. For instance, when \(\dot\gamma(t)=e^{\alpha(t)}\) is met, the associated Euler-Lagrange equations reduce to
\[
  \nabla^2\Phi\left( x+e^{-\alpha(t)} \dot x \right)
  \left[ \ddt\left(  x+e^{-\alpha(t)} \dot x \right) \right]
  + e^{\alpha(t)+\beta(t)}\nabla f(x) = 0 \,.
\]
The second condition ensures convergence of the underlying trajectories to the minimum at rate non-slower than \(\calO(e^{-\beta(t)})\)

In the particular case where \(\Phi(x)=\tfrac12\|x\|^2\), for which \(\calB_{\Phi}(x,y)=\tfrac12\|x-y\|^2\), the Bregman Lagrangian takes the simple form
\begin{equation}
  \label{eq:bregman.lagrangian:euclidean}
  L(x,\dot{x},t) = a(t)\tfrac12\|\dot{x}\|^2-b(t)f(x) \,,
\end{equation}
with \(a(t)=e^{\gamma(t)-\alpha(t)}\) and \(b(t)=e^{\alpha(t)+\beta(t)+\gamma(t)}\).

\section{Geometry of the Time-Dependent Lagrangian and Hamiltonian Formalisms}
\label{sec:cosym}

Since Bregman Lagrangians are time-dependent, in this section, we introduce some needed geometric ingredients about non-autonomous mechanics and highlight some of their main invariance properties \citep[see][]{Abraham1978, marle, deLeon1987}.

Let \(Q\) be a manifold and \(TQ\) its tangent bundle. Coordinates \((x^i)\) on \(Q\) induce coordinates \((x^i, \dot{x}^i)\) on \(TQ\). Therefore we have natural coordinates \((x^i, \dot{x}^i, t)\) on \(TQ\times \bbR\) which is the velocity phase space for time-dependent systems.

Given two instants (time values) \(a, b\in \bbR\), with \(a<b\), and corresponding positions \(x_a, x_b\in Q\), consider the set of curves:
\[
  \calC^2_{a,b} =
  \calC^2([a, b], x_a, x_b) =
  \{ \sigma\colon [a, b]\to Q\; |\;
  \sigma\in \calC^2 \text{ with } \sigma(a)=x_a, \; \sigma(b)=x_b\} \,.
\]
Given a time-dependent Lagrangian function \(L\colon TQ\times\bbR\to\bbR\), define the action functional \(\calJ_L\colon\calC^2_{a,b}\to \bbR\)
\begin{equation}
  \label{eq:action}
  \calJ_L(\sigma) = \int_a^b L(\sigma'(t),t)\,\dt
\end{equation}
where \(\sigma'\colon [a, b]\to TQ\).

Using variational calculus, the critical points of \(\calJ_L\) are locally characterized by the solutions of the \textbf{Euler-Lagrange equations}:
\begin{equation}
  \label{eq:el}
  \ddt\left(\pp[L]{\dot x^i}\right)-\pp[L]{x^i}=0\,, \quad 1\leq i\leq n=\dim Q \,.
\end{equation}
For time-dependent Lagrangians it is possible to check that the energy \(E_L\colon TQ\times \bbR\to \bbR\),
\[
  E_L=\Delta L-L=\dot{x}^i\frac{\partial L}{\partial \dot{x}^i}-L \,,
\]
where \(\Delta\) is the Liouville vector field on \(TQ\) \citep{marle}, is not, in general, preserved since
\[
  \dd[E_L]t=\pp[L]t \,.
\]

We now pass to the Hamiltonian formalism using the \textbf{Legendre transformation}
\[
\calF L\colon TQ\times\bbR\To T^*Q\times\bbR \,,
\]
where \(T^*Q\) is the cotangent bundle of \(Q\) whose natural coordinates are \((x^i, p_i)\). The Legendre transformation is locally given by
\[
  \calF L (x^i, \dot{x}^i, t) = \left(x^i, \pp[L]{\dot x^i}, t\right) \,.
\]

We assume that the Legendre transformation is a diffeomorphism (that is, the Lagrangian is hyperregular) and define the Hamiltonian function \(H\colon T^*Q\times \bbR\to \bbR\) by
\[
  H=E_L\circ (\calF L)^{-1} \,,
\]
which induces the cosymplectic structure \((\Omega_H, \eta_\bbR)\) on \(T^*Q\times \bbR\) with
\[
  dt\vcentcolon=\pr_2^*\dt\; , \qquad \Omega_H=-\diff(\pr_1^*\theta_Q-H dt)=\Omega_Q+\diff H\wedge dt \,,
\]
where \(\pr_i\), \(i=1,2\), are the projections to each Cartesian factor and \(\theta_Q\) denotes the Liouville 1-form on \(T^*Q\) \citep{Abraham1978}, given in induced coordinates by \(\theta_Q=p_i\, \dx^i\). We also denote by \(\Omega_Q=-\diff\pr_1^*\theta_Q\) the pullback of the canonical symplectic 2-form \(\omega_Q=-\diff\theta_Q\) on \(T^*Q\). In coordinates, \(\Omega_Q=\dx^i\wedge \diff p_i\). (Observe that now \(\Omega_Q\) is presymplectic since \(\ker \Omega_Q=\text{span}\{\partial/\partial t\}\).) Therefore in induced coordinates \((x^i, p_i, t)\):
\[
  \Omega_H=\dx^i\wedge \diff p_i+\diff H\wedge\dt \,,
  \qquad
  \eta_\bbR=\dt \,.
\]

We define the \textbf{evolution vector field} \(E_H\in \mathfrak{X}(T^*Q\times \bbR)\) by
\begin{equation}
  \label{eq:qqq}
	i_{E_H}\Omega_H=0\; ,\qquad i_{E_H}dt=1
\end{equation}
In local coordinates the evolution vector field is:
\[
  E_H=\ppt+\pp[H]{p_i}\pp{x^i}-\pp[H]{x^i}\pp{p_i} \,.
\]
The integral curves of \(E_H\) are given by:
\begin{equation}\label{eq:hamiltonian}
  \dot t=1\,,\qquad \dot x^i=\pp[H]{p_i} \,, \qquad \dot p_i=-\pp[H]{x^i}\,.
\end{equation}
The integral curves of \(E_H\) are precisely the curves of the form \(t\mapsto\calF L(\sigma'(t),t)\) where \(\sigma\colon I\to Q\) is a solution of the Euler-Lagrange equations for \(L\colon TQ\times \bbR\to \bbR\).

From Equation \eqref{eq:qqq} we deduce that the flow of \(E_H\) verifies the following preservation properties
\begin{equation}
  \label{eq:kop}
 \calL_{E_H}\Omega_H = \calL_{E_H}( \Omega_Q+\diff H\wedge dt) = 0 \,,
  \qquad
  \calL_{E_H} dt = 0 \,.
\end{equation}
Denote by \(\Psi_s\colon \calU\subset T^* Q\times \bbR\To T^* Q\times \bbR\) the flow of the evolution vector field \(E_H\), where \(\calU\) is an open subset of \(T^*Q\times \bbR\).
Observe that
\[
  \Psi_s(\alpha_q,t) = (\Psi_{t,s}(\alpha_q),t+s) \,,
  \ \alpha_q\in T_q^* Q \,,
\]
where \(\Psi_{t,s}(\alpha_q)=\pr_1(\Psi_s(\alpha_q, t))\).
Therefore from the flow of \(E_H\) we induce a map
\[
\Psi_{t,s}\colon\calU_t\subseteq T^*Q\to T^*Q
\]
where \(\calU_t=\{ \alpha_q\in T^*Q \;|\; (\alpha_q,t)\in \calU \}\). Observe that if we know \(\Psi_{t,s}\) for all \(t\), we can recover the flow \(\Psi_s\) of \(E_H\).

From Equations \eqref{eq:kop} we deduce that
\begin{equation}\label{eq:preser}
  \Psi_s^*(\Omega_Q+\diff H\wedge dt)=\Omega_Q+\diff H\wedge dt \,,
  \qquad
	\Psi_s^*(dt)=dt \,.
\end{equation}
The following theorem relates the  preservation properties (\ref{eq:preser}) with the symplecticity of the map family \(\{\Psi_{t,s}\colon T^*Q\to T^*Q\}\).

\begin{theorem}
  \label{thm:wer}
  We have that \(\Psi_{t,s}\colon \calU_t\subseteq T^*Q\to T^*Q\) is a symplectomorphism, that is, \(\Psi_{t,s}^*\omega_Q=\omega_Q\).
\end{theorem}

\begin{proof}
   First, observe  that any vector \(Y_{(\alpha_q,t)} \in T_{(\alpha_q, t)} (T^*Q\times \bbR)\) admits a unique decomposition:
  \[
    Y_{(\alpha_q,t)}=Y_{\alpha_q}(t)+Y_t(\alpha_q) \,,
  \]
  where \(Y_{\alpha_q}(t)\in T_{\alpha_q}T^*Q\) and \(Y_t (\alpha_q)\in T_t\bbR\).
  Moreover, we have that  \(\langle dt, Y_{\alpha_q}(t)\rangle=0\).

  Therefore, if we restrict ourselves to vectors tangent to the \(\pr_2\)-fibers  \(Y_{(\alpha_q,t)}\in T_{(\alpha_q, t)} \pr_2^{-1}(t)=V_{(\alpha_q,t)}\pr_2\) then we have the decomposition
  \[
    Y_{(\alpha_q,t)}=Y_{\alpha_q}(t)+0_t=Y_{\alpha_q}(t)\in V_{(\alpha_q,t)}\pr_2\equiv T_{\alpha_q}T^* Q \,.
  \]
  From the second preservation property given in \eqref{eq:kop} we deduce that
  \[
    0=\langle (dt)_{(\alpha_q, t)}, Y_{\alpha_q}(t)\rangle=\langle (\Psi_s^*dt)_{(\alpha_q, t)}, Y_{\alpha_q}(t)\rangle =\langle (dt)_{\Psi_s(\alpha_q, t)}, T\Psi_s (Y_{\alpha_q}(t))\rangle
  \]
  Therefore \(T\Psi_s (Y_{\alpha_q}(t))\in V_{(\Psi_{s, t}(\alpha_q),t+s)}pr_2\equiv T_{\Psi_{t,s}(\alpha_q)}T^*Q\) and
  \[
    T\Psi_s (Y_{\alpha_q}(t))=T\Psi_{t,s} (Y_{\alpha_q}(t))+0_{t+s}\equiv T\Psi_{t,s} (Y_{\alpha_q}(t)) \,.
  \]
 Using the first identity in \eqref{eq:kop} we deduce that
  \begin{align*}
    (\omega_Q)_{\alpha_q}(Y_{\alpha_q}(t), \tilde{Y}_{\alpha_q}(t))
    ={}&(\Omega_Q+\diff H\wedge\dt)_{(\alpha_q, t)} (Y_{\alpha_q}(t), \tilde{Y}_{\alpha_q}(t))\\
    ={}&(\Omega_Q+\diff H\wedge\dt)_{(\Psi_{s,t}(\alpha_q), t+s)} (T\Psi_s (Y_{\alpha_q}(t)), T\Psi_s (\tilde{Y}_{\alpha_q}(t)))\\
    ={}&(\omega_Q)_{\Psi_{s,t}(\alpha_q)} (T\Psi_{t,s} (Y_{\alpha_q}(t)), T\Psi_{t,s} (\tilde{Y}_{\alpha_q}(t)))
	\end{align*}
  where \(Y_{\alpha_q}(t), \tilde{Y}_{\alpha_q}(t)\in T_{\alpha_q}T^*Q\equiv T_{(\alpha_q, t)}\pr_2^{-1}(t)\). We conclude that \(\Psi_{t,s}^*\omega_Q=\omega_Q\).
\end{proof}

\section{Discrete Variational Methods for Time-Dependent Lagrangian Systems}
\label{sec:del}

Consider the set of discrete paths (or sequences) on \(Q\) for a fixed number of steps  \(N\in\bbN\), that is, the set
\[
  \calC_d(Q)=\left\{ x_d\colon\left\{0,1,\dots,N\right\} \to Q \right\} = Q\times \stackrel{(N+1)}{\cdots} \times Q \,.
\]
Then an appropriate discrete interpretation of the velocities are pairs in \(Q\times Q\), a discrete version of \(TQ\).

A \textbf{discrete time-dependent Lagrangian} is a family of maps
\[
  L^k_d\colon Q\times Q\to \bbR, \qquad k\in\bbN \,,
\]
for which we define the \textbf{discrete action map} on the space of sequences as
\[
	S_d (x_d)=\sum_{k=0}^{N-1}L^k_d(x_k,x_{k+1}), \quad x_d\in \calC_d(Q) \,.
\]
If we consider variations of \(x_d\) with fixed end points \(x_0\) and \(x_N\) and extremize \(S_d\) over \(x_1,\ldots,x_{N-1}\), we obtain the \textbf{discrete Euler-Lagrange equations} (DEL for short)
\[
	\partial_{x_k} S_d (x_d)=D_1 L^k_d (x_k,x_{k+1})+D_2L^{k-1}_d(x_{k-1},x_k)=0 \ \mbox{ for all } \ k=1,\ldots,N-1 \, .
\]
where \(D_1\) and \(D_2\) denote the partial derivatives with respect to the first and second components, respectively.

If, for all \(k\), \(L^k_d\) is regular, that is, the matrix
\[
	D_{12}L^k_d =\left(\pp[{}^2 L_d^k]{x_k\partial x_{k+1}}\right)
\]
is non-singular, then we locally obtain a well defined family of discrete Lagrangian maps:
\[
	\begin{array}{cccc}
    F_{k,k+1}\colon & Q\times Q & \To & Q \times Q \\
                    & (x_k,x_{k+1}) & \longmapsto & (x_{k+1},x_{k+2}(x_k,x_{k+1},k)) \, .
	\end{array}
\]
where the value of \(x_{k+2}\) is determined in terms of \(x_k\), \(x_{k+1}\) and \(k\).
In this setting, we can define two discrete Legendre transformations associated to \(L^k_d\), \(\bbF^\pm L^k_d\colon Q\times Q\to T^*Q\), by the expressions
\begin{align*}
  \bbF^+L^k_d\colon (x_k,x_{k+1}) \longmapsto{}& (x_{k+1},D_2L^k_d(x_k,x_{k+1})) \,,\\
  \bbF^-L^k_d\colon (x_k,x_{k+1}) \longmapsto{}& (x_k,-D_1L^k_d(x_k,x_{k+1})) \,.
\end{align*}

We can also define the evolution of the discrete system on the Hamiltonian side, \(\tilde{F}_{k, k+1}\colon T^*Q\to T^*Q\), by any of the formulas:
\[
	\tilde{F}_{k,k+1}
  = \bbF^+L^k_d\circ (\bbF^-L^k_d)^{-1}
  = \bbF^+L^k_d\circ F_{k-1,k} \circ (\bbF^+L^{k-1}_d)^{-1}
  = \bbF^-L^{k+1}_d\circ F_{k,k+1} \circ (\bbF^-L^k_d)^{-1} \,,
\]
because of the commutativity of the following diagram.
\[
	\xymatrix{
		(x_{k-1},x_k) \ar[rr]^{F_{k-1,k}} \ar[dr]_{\bbF^+L^{k-1}_d} & & (x_k,x_{k+1}) \ar[ld]_{\bbF^-L^k_d} \ar[rd]_{\bbF^+L^k_d} \ar[rr]^{F_{k, k+1}} & & (x_{k+1},x_{k+2}) \ar[ld]_{\bbF^-L^{k+1}_d} \\
		& (x_k,p_k) \ar[rr]_{\tilde{F}_{k,k+1}} & & (x_{k+1},p_{k+1}) &
	}
\]

\begin{proposition}
	The discrete Hamiltonian map \(\tilde{F}_{k, k+1}\colon(T^*Q,\omega_Q) \To (T^*Q,\omega_Q)\) is a \textbf{symplectic transformation}, that is
  \[
    (\tilde{F}_{k, k+1})^*\omega_Q=\omega_Q\; .
  \]
\end{proposition}
\begin{proof}
  Using similar arguments to the autonomous case \citep{marsden-west}, we deduce that
  \[
    (\bbF^+L^k_d)^*\omega_Q=(\bbF^-L^k_d)^*\omega_Q \,.
  \]
  From the definition of \(\tilde{F}_{k,k+1}: T^*Q\rightarrow T^*Q\) we deduce that
  \[
    (\tilde{F}_{k,k+1})^*\omega_Q
    = (\bbF^+L^k_d\circ (\bbF^-L^k_d)^{-1})^*\omega_Q
    = ((\bbF^-L^k_d)^*)^{-1}((\bbF^+L^k_d)^*\omega_Q)
    = \omega_Q \,.
  \]
\end{proof}

Given the map \(\tilde{F}_{k, k+1}(q_k, p_k)=(q_{k+1}, p_{k+1})\), we immediately have the map
\[
  (x_k, p_k,kh)\mapsto(x_{k+1}, p_{k+1},(k+1)h)
\]
on \(T^*Q\times\bbR\) where we now give explicit information of the evolution of discrete time.

Let see the relation of the of these discrete maps \(F_{k,k+1}\colon Q\times Q\to Q\times Q\) and \(\tilde{F}_{k, k+1}\colon T^*Q\to T^*Q\)
with the Euler-Lagrange equations and Hamilton equations of a time-dependent Lagrangian system. Given a regular Lagrangian function \(L\colon TQ\times \bbR \to \bbR\) and a sufficiently small time step \(h>0\), we are going to define an \(h\)-and-\(k\)-dependent family of discrete Lagrangian functions \(L^k_{d,h}\colon Q\times Q \to \bbR\) as an infinitesimal approximation to the continuous action \(\calJ_L\) defined in expression \eqref{eq:action}. As intermediate step, we first consider the \textbf{exact time-dependent discrete Lagrangian} associated to a regular Lagrangian \(L\) which is given by the expression
\[
  L^{k,E}_{d,h}(x_0,x_1) = \frac1h\int_{kh}^{(k+1)h} L(x_{0,1}(t),\dot x_{0,1}(t),t)\,\dt \,,
\]
where \(x_{0,1}(t)\) is the unique solution of the Euler-Lagrange equations for \(L\) satisfying \(x_{0,1}(kh)=x_0\) and \(x_{0,1}((k+1)h)=x_1\), \citep[see][]{hartman,MMdDM2016}. Then for a sufficiently small \(h\), the solutions of the DEL for \(L^{k,E}_{d,h}\) lie on the solutions of the Euler-Lagrange equations for \(L\), see \cite[Theorem 1.6.4]{marsden-west}.

In practice, \(L^{k,E}_{d,h}(x_0,x_1)\) will not be available, therefore we take an approximation,
\[
	L^k_{d,h}(x_0,x_1) \approx L^{k,E}_{d,h}(x_0,x_1)\, ,
\]
using some quadrature rule. Then, as we have seen, the scheme derived from the DEL will be geometric integrators for Equations (\ref{eq:hamiltonian}) preserving the symplectic form in the sense of Theorem \ref{thm:wer} \citep[see][]{PatrickCuell}.

\begin{remark}
  As we have commented on the introduction, one of the main advantages of the proposed approach is the possibility to use other options to derive different numerical methods for optimization by only discretizing a unique function, the action functional. Of course, there are many different ways to do it \citep{marsden-west}. For instance, we can combine several discrete Lagrangians together to obtain a new discrete Lagrangian with higher order (composition methods) or similarly obtaining splitting methods \citep{CamposSS17}. Also we can easily derive symplectic partitioned Runge-Kutta methods or symplectic Garlekin methods using polynomial approximations to the trajectories and a numerical quadrature to approximate the action integral \citep{Campos14}. Moreover, it is possible to adapt the variational integrators to a non-euclidean setting using appropriate retraction maps.
\end{remark}

\section{Discretization of Lagrangian Systems with Forces}
\label{sec:forces}

Our intention here is to continue looking for numerical approximations to the time-dependnet Euler-Lagrange equations  but considering additionally an external force that decreases jointly with the time-step parameter \(h\). With it we will obtain a whole family of algorithms whose behavior resembles that of the Nesterov method.
Fortunately, discrete mechanics is also adapted to the case of external forces \citep[see][]{marsden-west}. To this end, in addition to a time-dependent Lagrangian function \(L\colon TQ\times \bbR\to \bbR\), we have an external force given by a fiber preserving map \(F\colon TQ\times \bbR \to T^*Q\) given locally by
\[
  F(x,\dot x,t)=(x ^i, F_i(x,\dot x,t))
\]
Given the force \(f\), we derive the equations of motion of the forced system modifying the Hamilton’s principle to the \textbf{Lagrange-d’Alembert principle}, which seeks curves \(\sigma\in\calC^2_{a, b}\) satisfying
\begin{equation}\label{ldp}
  \delta \int_a^b L(\sigma'(t), t)\;\dt+\int_a^b F(\sigma'(t), t)\delta \sigma(t)\;\dt=0\; ,
\end{equation}
for all \(\delta \sigma\in T_{\sigma}\calC^2_{a, b}\).
Using integration by parts we derive the forced Euler-Lagrange equations, which have the following coordinate expression:
\[
  \ddt\left( \pp[L]{\dot x^i}\right) -\pp[L]{x^i}=F_i \,.
\]

To discretize these equations we consider as before a family of Lagrangian functions \(L_d^k\colon Q\times Q\to \bbR\) and two discrete forces \((F^k_d)^\pm\colon Q \times Q\to T^*Q\), which are fiber preserving in the sense that \(\pi_Q \circ (F^k_d)^\pm = \pr_\pm\), where \(\pr_\pm(x_-, x_+)=x_\pm\). Combing both forces we obtain \(F^k_d\colon Q\times Q\to T^*(Q\times Q)\) by
\[
  \langle F^k_d(x_k, x_{k+1}), (\delta x_k, \delta x_{k+1})\rangle= (F^k_d)^-(x_k, x_{k+1})\delta x_k+(F^k_d)^+(x_k, x_{k+1})\delta x_{k+1} \,.
\]
As in \eqref{ldp} we have a discrete version of the Lagrange-d'Alembert principle for the discrete forced system given by \(L^k_d\) and \(F^k_d\):
\[
  \delta \sum_{k=0}^{N-1} L_d^k(x_k,x_{k+1})
  + \sum_{k=0}^{N-1} \langle F^k_d(x_k,x_{k+1}) , (\delta x_k,\delta x_{k+1}) \rangle
  = 0 \,,
\]
for all variations \(\{\delta x_k\}_{k=0}^N\) vanishing at the endpoints, that is, \(\delta x_0=\delta x_N=0\).  This is equivalent to the forced discrete Euler-Lagrange equations:
\begin{equation}\label{eq:forcedsystem}
  D_1L^k_d(x_k,x_{k+1}) + D_2L^{k-1}_d(x_{k-1},x_k)
  + (F^k_d)^-(x_k, x_{k+1}) +(F^{k-1}_d)^+(x_{k-1}, x_k)
  = 0 \,,
\end{equation}
for all \(k=1,\dots,N-1\).

\section{The Variational Derivation of PHB/CM and NAG}
\label{sec:cm-nag}

As seen in Section \ref{sec:from-GD-to-NAG}, NAG can be derived naively from PHB/CM.
Besides, under suitable conditions on \(\mu,\eta\) and the starting points,
both methods converge to a minimum of \(f\), the latter, NAG, doing so faster
\citep{Po64,Nesterov}. Questions arise: What makes NAG faster than PHB/CM? Can this
be exploited to obtain even faster methods? Can it be generalized? Questions
that boil down to how NAG is fundamentally derived from PHB/CM.

To begin with, note that the NAG equations \eqref{eq:nag} can be rewritten only
in terms of the \(x\)'s, as in \eqref{eq:cm:single}, or only in terms of the
\(y\)'s, yielding the equations
\begin{subequations}
  \label{eq:nag:xing-yang}
\begin{align}
  \label{eq:nag:xing}
  \Delta \bar x_k ={}& \mu_{k\phantom{\!-1}} \Delta \bar x_{k-1} - \eta_k\nabla f(\bar x_k) - \mu_k \big(\eta_k\nabla f(\bar x_k) - \eta_{k-1}\nabla f(\bar x_{k-1}) \big)\\
  \label{eq:nag:yang}
  \Delta \bar y_k ={}& \mu_{k-1}\Delta \bar y_{k-1} - \eta_k\nabla f(\bar y_k+\mu_{k-1}\Delta \bar y_{k-1})
\end{align}
\end{subequations}
The first, Eq. \eqref{eq:nag:xing}, when compared to \eqref{eq:cm:single} shows
an extra term, \[\mu_k \big(\eta_k\nabla f(\bar x_k) - \eta_{k-1}\nabla f(\bar
  x_{k-1}) \big)\; ,\] that in fact points to the very origin of the method: an
additional forcing term. The second, Eq. \eqref{eq:nag:yang}, compared again to
\eqref{eq:cm:single} shows that the \(y\)-trajectory is obtained almost as if it
was computed by PHB/CM but evaluating \(\nabla f\) at a ``future'' point, \(\bar
y_k+\mu_{k-1}\Delta \bar y_{k-1}\), which ``informs better'' the method on how
to advance towards the minimum.

Besides of the convergence towards the minimum, it can be shown that both
methods are a time discretization of the second order differential equation
\citep{SuBoCa16}
\begin{equation}
  \label{eq:sode-cm}
  \ddot x + \nu(t)\dot x + \eta(t)\nabla f(x) = 0 \,,
\end{equation}
a well known fact in the literature, Equation that furthermore is in general
variational (see Lemma \ref{thm:sodevar}). A fact that is not so well known is
that NAG better discretizes the equation when including a force term
proportional to the underlying time step and, moreover, it can be derived, as
well as PHB/CM, from a variational approach, in the geometric integration sense (see
Sections \ref{sec:del} and \ref{sec:forces}).

We first give a rather simple and direct result whose purpose is to establish
properly the continuous setting over which the latter discretizations will be
built and from which the methods can be derived.

\begin{lemma}
  \label{thm:sodevar}
  Given a vector field \(P\colon\bbR^n\to\bbR^n\), consider the second order
  differential equation
  \begin{equation}
    \label{eq:sode-nesterov}
    \ddot x + \nu(t)\dot x + \eta(t)P(x) 
    = \varepsilon
    \ddt\Big[\eta(t)P(x) \Big] \,,
  \end{equation}
  where \(\nu,\eta\colon\bbR_+\to\bbR\) are continuous time-dependent real
  valued functions and where \(\varepsilon\in\bbR\) is a constant.
  If \(P\) is conservative, that is, if \(P=\nabla f\), then
  \eqref{eq:sode-nesterov} corresponds to the equation of motion of the (forced)
  time-dependent Lagrangian system:
  \begin{subequations}
    \label{eq:csys}
  \begin{align}
    \label{eq:csys-lagrangian}
    L(x,\dot x,t) ={}& a(t)\tfrac12\|\dot x\|^2-b(t)f(x) \,,\\
    \label{eq:csys-force}
    F(x,\dot x,t) ={}& \varepsilon a(t)
                       \ddt\left[ \frac{b(t)}{a(t)}P(x) \right] \,,
  \end{align}
  \end{subequations}
  in which \(f\) is the field's potential and where
  \begin{equation}
    \label{eq:coeffs-var2lag}
     a(t) = \exp(\textstyle\int_0^t\nu(s)\:\!\ds)\,, \quand b(t) = a(t)\eta(t) \,,
  \end{equation}
  for \(t\geq0\).
\end{lemma}

\begin{proof}
  Assume \(P\) is conservative and let \(f\) denote it's potential. Then the
  Euler-Lagrange equation for \eqref{eq:csys} is
  \begin{equation}
    \label{eq:EL-forces}
    a(t)\ddot x+a'(t)\dot x + b(t)P(x)
    = \varepsilon a(t)
    \ddt\left[ \frac{b(t)}{a(t)}P(x) \right] \,.
  \end{equation}
  Dividing by \(a(t)\), and taking into account that, from
  \eqref{eq:coeffs-var2lag} we have that
  \begin{equation}
    \label{eq:coeffs-lag2var}
    \nu(t) = \frac{a'(t)}{a(t)} \quand \eta(t) = \frac{b(t)}{a(t)}\,,
  \end{equation}
  we obtain \eqref{eq:sode-nesterov}.
\end{proof}

\begin{remark}
  \(P\) being conservative is not a necessary condition for
  \eqref{eq:sode-nesterov} to be derived from the Lagrange-d'Alembert principle.
  In order to be variational, Equation \eqref{eq:EL-forces} requires a
  Lagrangian of the form
  \[ L(x,\dot x,t) = a(t)\tfrac12\dot x^2 + \langle c(x,t), \dot x\rangle + d(x,t) \]
  for some unknown functions \(c,d\), which implies
  \[ \pp[d]x = \pp[c]t + b(t)P(x) \,. \]
  A vector field \(P\) satisfying this last relation need not be conservative.
\end{remark}

Next, the result that links the previous continuous equation
\eqref{eq:sode-nesterov} with PHB/CM and NAG, showing, in particular, that NAG is a
forced version of PHB/CM, between which the transition is immediate.

\begin{theorem}
  \label{thm:cm-vs-nag}
  Given a real valued function \(f\colon\bbR^n\to\bbR\) and a vector field
  \(P\colon\bbR^n\to\bbR^n\), consider the time-dependent discrete Lagrangian
  and forces
  \begin{subequations}
    \label{eq:dsys}
  \begin{align}
    \label{eq:dsys-lagrangian}
    L^k_d(z_0,z_1) ={}& a_k\tfrac12\left\|z_1-z_0\right\|^2-b_k^-f(z_0)-b_{k+1}^+f(z_1) \,,\\
    \label{eq:dsys-forcem}
    (F_d^k)^-(z_0,z_1) ={} & {-\tfrac{a_{k-1}}{a_k}}(b_k^-+b_k^+)P(z_0)\,,\ \text{and}\\
    \label{eq:dsys-forcep}
    (F_d^k)^+(z_0,z_1) ={} & \phantom{-\tfrac{a_{k-1}}{a_k}}(b_k^-+b_k^+)P(z_0)\,.
  \end{align}
  \end{subequations}
  where \(\{a_k\}_{k\geq0}\), \(\{b_k^-\}_{k\geq0}\), \(\{b_k^+\}_{k\geq0}\),
  are arbitrary sequences of real numbers.
  If \(f\) is regular enough, so that \(P=\nabla f\), and \(a_k\) is never null,
  then the free and forced equations of motion for \(L^k_d\) and
  \((L^k_d,(F_d^k)^-,(F_d^k)^+)\) are, respectively, equivalent to the following recursive schemes
  \begin{subequations}
    \label{eq:cmnag}
  \begin{align}
    \label{eq:cmnag:tail}
    y_{k+1} ={}& x_k - \eta_kP(x_k)&
    \bar y_{k+1} ={}& \bar x_k - \eta_kP(\bar x_k)\\
    \label{eq:cmnag:head}
    x_{k+1} ={}& y_{k+1} + \mu_k(x_k-x_{k-1})&
    \bar x_{k+1} ={}& \bar y_{k+1} + \mu_k(\bar y_{k+1}-\bar y_k)
  \end{align}
  \end{subequations}
  where
  \begin{equation}
    \label{eq:coeffs-lag2int}
    \mu_{k+1}=\frac{a_k}{a_{k+1}} \quand
    \eta_k=\frac{b_k^-+b_k^+}{a_k} \,,
  \end{equation}
  for \(k\geq0\).

  Conversely, given a vector field \(P\colon\bbR^n\to\bbR^n\) and two arbitrary
  sequences of real numbers \(\{\mu_{k+1}\}_{k\geq0}\) and
  \(\{\eta_k\}_{k\geq0}\), consider the sequences of pairs of points given in
  equations \eqref{eq:cmnag:tail}-\,\eqref{eq:cmnag:head}.
  If \(P\) is conservative and \(\mu_{k+1}\) is never null, then both schemes
  are variational. Moreover, they are equivalent to the equations of motions for
  the free and forced time-dependent discrete Lagrangian systems given in
  \eqref{eq:dsys-lagrangian}-\,\eqref{eq:dsys-forcep}, for which \(f\) is the
  field's potential and
  \begin{equation}
    \label{eq:coeffs-int2lag}
    a_0 = 1,\
    a_{k+1} = a_k/\mu_{k+1},\ \forall k\geq0\,,\qquad
    b_k^\pm = \tfrac12 a_k\eta_k,\ \forall k\geq0\,.
  \end{equation}
\end{theorem}

\begin{proof}
  Partial differentiation of the Lagrangian gives
  \begin{align*}
    D_1L^k_d(z_0,z_1) ={}& -a_k\Delta z_0-b_k^-\nabla f(z_0)\\
    D_2L^k_d(z_0,z_1) ={}& \phantom{{}-{}}a_k\Delta z_0-b_{k+1}^+\nabla f(z_1)
  \end{align*}
  from where it is readily seen that the DEL equations with forces (\ref{eq:forcedsystem}) are
  \begin{multline*}
    -a_{k+1}\Delta z_1 + a_k\Delta z_0 - (b_{k+1}^-+b_{k+1}^+)\nabla f(z_1) =\\
    = \frac{a_k}{a_{k+1}}(b_{k+1}^-+b_{k+1}^+)\nabla f(z_1)
                         - (b_k^- + b_k^+)\nabla f(z_0)\,,
  \end{multline*}
  where the RHS is null for the non-forced DEL equations. Dividing by
  \(-a_{k+1}\) and using the relations \eqref{eq:coeffs-lag2int}, we get
  \begin{equation}
    \label{eq:DEL.simplified}
    \Delta z_1  -\mu_{k+1}\Delta z_0 + \eta_{k+1}\nabla f(z_1)
    = -\mu_{k+1}\left(\eta_{k+1}\nabla f(z_1)-\eta_k\nabla f(z_0)\right)\,,
  \end{equation}
  where again the right hand side is null for the non-forced case. Replacing
  \(z_i\) by \(x_{k+i}\), in the non-forced case, and by \(\bar x_{k+i}\), in
  the forced one, taking into account that \(P=\nabla f\), and using the
  equations in \eqref{eq:cmnag:tail}, we obtain those in \eqref{eq:cmnag:head}.

  The converse is immediate.
\end{proof}

Now remarks are in order that will summarize some points that have been
mentioned earlier and underline others that haven't been yet.

\begin{remark}[One-to-one correspondence]
  \label{rmk:one-to-one-and-onto}
  Note that both equations in \eqref{eq:cmnag:tail} are formally the same,
  whereas in \eqref{eq:cmnag:head} there is a difference in the last term: While
  PHB/CM uses \(x\)'s, NAG considers \(y\)'s. This is a slight change that
  nonetheless defines different schemes and in which the forcing term is hidden.
  This result not only shows that NAG is a forced version of PHB/CM, but that the
  schemes are in a natural bijective correspondence.
\end{remark}

\begin{remark}[Initial conditions]
  \label{rmk:initial.conditions}
  Usually the initial condition, \((x_0,v_0)\) or \((x_0,p_0)\) in phase space,
  or \((x_0,x_1)\) in configuration space, is crucial for the proper simulation
  of the dynamical system. Here however the dynamics are a tool and generally an
  initial condition so that \(x_1=x_0\) or \(\bar y_0=\bar x_0\), where
  \(x_0=\bar x_0\) is close to the minimum, will suffice, which corresponds to
  stick the ball to the bowl's wall and leave it roll.
\end{remark}

\begin{remark}[Natural trajectory]
  \label{rmk:trajectories}
  From the schemes' definitions, the sequences \(\{x_k\}_{k=0}^\infty\) and \(\{\bar x_k\}_{k=0}^\infty\) are the natural (on track) dynamical trajectories towards the minimum of \(f\), while \(\{y_k\}_{k=1}^\infty\) and \(\{\bar y_k\}_{k=1}^\infty\) are off road marks that limit these trajectories like slalom flags. The latter are however asymptotically close to the former and, hence, to the minimum.
\end{remark}

\begin{remark}[Discrete flow]
  \label{rmk:discrete.flow}
  If one wants to compare the trajectories obtained from both methods, perhaps Equations \eqref{eq:cmnag} are better suited. If one is solely interested in a simple implementation to compute the minimum, then Equations \eqref{eq:cm:single} and \eqref{eq:nag:yang} are a good alternative since they can easily be rewritten to give discrete flow updates in the form ``momentum first, then position'' as in \cite{pmlr-v28-sutskever13}:
 \begin{subequations}
  \label{eq:flow.updates}
    \begin{align}
      \label{eq:flow.updates:momemtum}
      \Delta x_k ={}& \mu_{k-1}\Delta x_{k-1} - \eta_k\nabla f(x_k)&
      \Delta\bar y_k ={}& \mu_{k-1}\Delta\bar y_{k-1} - \eta_k\nabla f(\bar y_k+\mu_{k-1}\Delta\bar y_{k-1})\\
      \label{eq:flow.updates:position}
      x_{k+1} ={}& x_k + \Delta x_k&
      \bar y_{k+1} ={}& \bar y_k + \Delta \bar y_k
    \end{align}
  \end{subequations}
  where both methods should be initialized with \(x_0=\bar y_0\) and \(\Delta x_{-1}=\Delta \bar y_{-1}=0\). Both approaches, Equations \eqref{eq:cmnag} and Equations \eqref{eq:flow.updates}, have been considered for the simulations of Section \ref{sec:simulations}.
\end{remark}

\begin{remark}[Force approximation]
  \label{rmk:force.approximation}
  The action induced by the discrete forces in \eqref{eq:dsys} is a second
  order approximation to the action induced by the continuous force
  \eqref{eq:csys-force}. Indeed, given continuous
  coefficients \(a(t),b(t)\), define \(a_k=a(kh)/h^2\) and \(b^\pm_k\) so that
  \(b^-_k+b^+_k=b(kh)\), and similarly for a continuous path \(x(t)\) and a variation \(\delta x(t)\) of it. Then
  \begin{multline*}
    h\sum_{k=0}^{N-1}F_d^k(x_k,x_{k+1})\cdot(\delta x_k,\delta x_{k+1}) =\\
    \begin{aligned}
      ={}& -h\sum_{k=1}^{N-1}\frac{a_{k-1}}{a_k}(b_k^-+b_k^+)\nabla f(x_k)\cdot\delta x_k+h\sum_{k=1}^{N-1}(b_{k-1}^-+b_{k-1}^+)\nabla f(x_{k-1})\cdot\delta x_k\\
      ={}& \sum_{k=1}^{N-1}\int_{t_k-\frac{h}2}^{t_k+\frac{h}2}\left(-\frac{a(t-h)}{a(t)}b(t)\nabla f(x(t))+b(t-h)\nabla f(x(t-h))\right)\cdot \delta x(t)\,\dt+\calO(h^2)\\
      ={}& \int_{\frac{h}2}^{T-\frac{h}2}-h\,a(t-h)\ddt\left[ \frac{b(t)}{a(t)}\nabla f(x(t)) \right]\cdot \delta x(t)\,\dt+\calO(h^2)\\
      ={}& \int_0^T-h\,a(t)\ddt\left[ \frac{b(t)}{a(t)}\nabla f(x(t)) \right]\cdot \delta x(t)\,\dt+\calO(h^2)
    \end{aligned}
  \end{multline*}
  where we have considered a mid point quadrature rule to establish the second
  equality, the rest follows.
\end{remark}

\begin{remark}[Force evanescence]
  \label{rmk:force.evanescence}
  As remarked, NAG can be viewed as an approximation to a forced continuous
  Lagrangian system, where the force is proportional to the time step that is
  used \emph{a posteriori} for the discretization. Although its contribution is
  non-null along the whole trajectory, it however vanishes not only locally,
  when \(h\) decreases, but also asymptotically in time, when \(t\) or \(k\)
  increase.
\end{remark}

\section{Simulations}
\label{sec:simulations}

Numerical experiments are performed considering different elements, namely:
\begin{itemize}
\item The time-dependent coefficients that appear in the Lagrangian, \(a(t)\)
  and \(b(t)\), for the simple case \eqref{eq:bregman.lagrangian:euclidean};
  \item The discretization scheme used to approximate the Lagrangian action; and obviously,
\item The objective function to be minimized, \(f(x)\).
\end{itemize}

\subsection{Lagrangian Coefficients}

We consider three different Lagrangians or, more precisely, three different pairs of time-dependent coefficients \((a(t),b(t))\) given, as in \eqref{eq:bregman.lagrangian:euclidean}, by time-dependent exponents triples \((\alpha(t),\beta(t),\gamma(t))\) satisfying the ideal scaling conditions \eqref{eq:ideal-conds}, which ensure that \(\argmin f\) is an attractor of the underlying dynamical system.

\subsubsection{Potential Dilation}

The first Lagrangian under consideration is a potential dilation of a mechanical Lagrangian, namely,
\begin{equation}
  \label{eq:clag-nesterov}
  L(x,v,t)
  = t^n\left( \frac12\|v\|^2 - f(x) \right) \,,
\end{equation}
whose Euler-Lagrange equation is
\begin{equation}
  \label{eq:EL-nesterov}
  \ddot x + \frac{n}t\dot x + \nabla f(x) = 0 \,,
\end{equation}
which is the equation considered in \cite{SuBoCa16}.

A naive choice of exponents from which to obtain the time-dependent coefficients \(a(t)=b(t)=t^n\) is
\begin{equation}
  \label{eq:exps-nesterov-nonideal}
  \alpha(t) = 0,\
  \beta(t)  = 0,\
  \gamma(t) = n\log t.
\end{equation}
However, they do not satisfy the ideal scaling conditions. A triple that does met this requirement is
\begin{equation}
  \alpha(t) = \log \powerp - \log t,\
  \beta(t)  = 2(\log t - \log \powerp),\
  \gamma(t) = \powerp\log t + \log \powerp,
\end{equation}
where \(\powerp=n-1\), but only when \(n\geq3\), thus, showing the ``magic'' \(n=3\) in \eqref{eq:EL-nesterov} of NAG \citep{SuBoCa16}. It is in fact the only choice satisfying the scaling conditions and gives an optimal rate of convergence non-slower than \(\calO(\sfrac1{t^2})\).

\subsubsection{Modified Potential Dilation}

The Lagrangian
\begin{equation}
  \label{eq:clag-wiwijo}
  L(x,v,t)
  = t^n\left( \frac12\|v\|^2 - Dt^{n-3}f(x)\right) \,,
\end{equation}
whose Euler-Lagrange equation is
\begin{gather*}
  \ddot x + \frac{n}t\dot x + Dt^{n-3}\nabla f(x) = 0 \,,
\end{gather*}
is the Lagrangian considered by \citet{WiWiJo16} for the metric case
and corresponds to the exponents
\begin{equation}
  \alpha(t) = \log \powerp - \log t,\
  \beta(t)  = \powerp\log t + \log C,\
  \gamma(t) = \powerp\log t + \log \powerp,
\end{equation}
where \(C=D/\powerp^2\) and \(\powerp=n-1\). They satisfy the ideal scaling conditions for \(n>1\), giving a rate of convergence non-slower than \(\calO(\sfrac1{t^{n-1}})\) \citep[confer with][, in particular for specific details on the constant \(C\)]{WiWiJo16}.

\subsubsection{Exponential Dilation}
Finally, the exponentially dilated Lagrangian
\begin{equation}
  \label{eq:clag-constant}
  L(x,v,t)
  = e^{\lambda t}\left( \frac12\|v\|^2 - f(x) \right) \,,
\end{equation}
whose equation of motion is precisely the one of a mechanical system with linear damping,
\begin{equation}
  \label{eq:EL-constant}
  \ddot x + \lambda\dot x + \nabla f(x) = 0 \,,
\end{equation}
corresponds to the choice exponents
\begin{equation}
  \label{eq:ideal_exponents:exp_dilation}
  \alpha(t) = \log\lambda,\
  \beta(t)  = -2\log\lambda,\
  \gamma(t) = \lambda t+\log\lambda.
\end{equation}
We note here three points. First, it is the unique choice for \(a(t)=b(t)=e^{\lambda t}\) that satisfies the ideal scaling conditions. Second, in principle this choice gives a theoretical convergence rate of \(\calO(1)\), fortunately it can be reduced down to \(o(1)\) \citep[, Th. 2.5]{Attouch21}. Third and more notably, although the Lagrangian \eqref{eq:clag-constant} is explicitly time-dependent, the Euler-Lagrange equation \eqref{eq:EL-constant} is autonomous and it introduces a linear time-independent damping term in the equation.

\subsection{Discretizations}
Using the trapezoidal rule to approximate the action of the above Lagrangians, we retrieve common NAG coefficients that appear in the literature. With some abuse of notation, we write in general
\[
  a(k) \coloneqq \frac{a(t_k)+a(t_{k+1})}{2h^2}\,,
  \qquand
  b^\pm(k) \coloneqq \frac{b(t_k)}2\,,
\]
where \(t_k=kh\).

\subsubsection{Bounded Coefficients from the Potential Dilation}

For the continuous time-dependent coefficients \(a(t)=b(t)=t^n\), with
\(n=\powerp+1\), the trapezoidal rule yields the discrete time-dependent coefficients
\[
  a(k) = \frac{t_k^n+t_{k+1}^n}{2h^2},\quad
  b^\pm(k) = \frac{t_k^n}2,
\]
from which we obtain the coefficients
\begin{equation}
  \label{eq:coeffs:bounded}
  \mu(k) = \frac{k^n + (k-1)^n}{k^n + (k+1)^n},\quad
  \eta(k) = \frac{2k^n}{k^n + (k+1)^n}h^2.
\end{equation}
Both coefficients are strictly increasing and bounded above by \(1\) and \(h^2\),
respectively. To avoid integer overflow and slightly reduce computational cost in
final implementations, these expressions can be simplified to
\begin{equation*}
  \mu(k) = \frac{2k-n}{2k+n} + o(1),\quad
  \eta(k) = \left(\frac{2k}{2k+n} + o(1)\right)h^2.
\end{equation*}
For the particular case \(n=3\), that is \(\powerp=2\), we get \(\mu(k+\sfrac32) =
\frac{k}{k+3} + \calO(\sfrac1{k^3}) \) as in \eqref{eq:nag:magic}.

\subsubsection{Unbounded Coefficients from the Modified Potential Dilation}

In this case, the continuous time-dependent coefficients are
\(a(t)=t^n\) and \(b(t)=Dt^{2n-3}\), as in \cite{WiWiJo16}, which yield
\begin{equation}
  \label{eq:coeffs:unbounded}
  \mu(k) = \frac{k^n + (k-1)^n}{k^n + (k+1)^n},\quad
  \eta(k) = D\frac{2k^n}{k^n + (k+1)^n}t_k^{n-3}h^2.
\end{equation}
As earlier,
\begin{equation*}
  \mu(k) = \frac{2k-n}{2k+n} + o(1),\quad
  \eta(k) = D\left(\frac{2k}{2k+n} + o(1)\right)t_k^{n-3}h^2.
\end{equation*}
This time, \(\eta\) is unbounded when \(n\geq4\), and bounded above by \(Dh^2\)
when \(n=3\).

Similarly, \citet{BJW} suggest a palindromic split Hamiltonian method with 7
stages that can be recovered from the proposed perspective considering the
discrete Lagrangian
\[ L^k_d(x_k,x_{k+1}) = t_{k+\sfrac12}^n\left(
    \frac12\left\|\frac{x_{k+1}-x_k}h\right\|^2 -
    Dt_{k+\sfrac12}^{n-3}\frac{f(x_k)+f(x_{k+1})}2 \right) \]
Note that it is not obtained by a trapezoidal approximation of the Lagrangian
\eqref{eq:clag-wiwijo} but it is still a discretization of it for which
\begin{equation*}
  \mu(k) = \left(\frac{2k-1}{2k+1}\right)^n,\quad
  \eta(k) = D\frac{(2k+1)^{2n-3} + (2k-1)^{2n-3}}{2(2k+1)^{2n-3}} t_{k+\sfrac12}^{n-3}h^2.
\end{equation*}
As before,
\[ \eta(k) = D\left(\frac{2k}{2k+2n-3}+o(1)\right)t_{k+\sfrac12}^{n-3}h^2 \,. \]

\subsubsection{Constant Coefficients from the Exponential Dilation}

Taking \(a(t)=b(t)=e^{\lambda t}\) yields
\begin{equation}
  \label{eq:coeffs:constant}
  \mu(k) = \frac{1 + e^{-\lambda h}}{1 + e^{+\lambda h}},\quad
  \eta(k) = \frac2{1 + e^{\lambda h}}h^2.
\end{equation}
For \(\lambda=1\) and \(h=0.1024\), \(\mu\approx0.9\) and \(\eta\approx0.01\),
values that often appear in the literature, as in \citet{pmlr-v28-sutskever13}. In general, any pair of constant coefficients \(\mu,\eta>0\) can be obtained from values \(\lambda\in\bbR,h>0\).

\subsubsection{The Actual Method by \texorpdfstring{\citeauthor*{WiWiJo16}}{Wibisono, Wilson, and Jordan}}

Strictly speaking, the method by \citet{WiWiJo16} is based on NAG, but it
differs from it. Although they consider the previous Lagrangian
\eqref{eq:clag-wiwijo} and experiment with it in \citet{BJW,Jo18}, what is
proposed in \citet{WiWiJo16} is the 3-phase scheme
\begin{subequations}
  \label{eq:WiWiJo16}
\begin{align}
  x_{k+1} ={}& \frac{\powerp}{k+\powerp}z_k + \frac{k}{k+\powerp} y_k\\
  \label{eq:argminyp}
  y_k ={}& \argmin_y\left\{ f_{\powerp-1}(y;x_k) +\frac{N}{\powerp h^\powerp}\|y-x_k\|^\powerp\right\}\\
  \label{eq:argminzp}
  z_k ={}& z_{k-1} - D\frac{k}{\powerp}t_k^{\powerp-2}h^2\nabla f(y_k)
\end{align}
\end{subequations}
with \(z_0=y_0=x_0\) and where \(\powerp=n-1\), \(f_{\powerp-1}(y;x_k)\) is the \((\powerp-1)\)-th
Taylor expansion of \(f\) about \(x_k\) and \(N\) is a constant related to \(D\)
and \(\powerp\) that ensures convergence. Note that for \(n=3\) (\(\powerp=2\)),
the optimization problem \eqref{eq:argminyp} is explicit and reduces to
\begin{equation}
  \label{eq:argminy2}
  y_k = x_k - \frac1{N}h^2\nabla f(x_k)
\end{equation}
but it is implicit in general, what increases the cost of the method, aside of having to compute the Hessian and higher derivatives of \(f\), either explicitly or by autodifferentiation.

In section \eqref{sec:experiments}, we will refer to method \eqref{eq:WiWiJo16} by WWJ, after the authors, and consider it as a (modified) NAG method.

\subsection{Objective Functions}
\label{sec:simulations:objective_functions}

Several objective functions are considered: a highly dimensional quadratic
function with tridiagonal matrix representation, a generalized Rosenbrock
function, yet another test function for momentum-descent methods and one that
combines a generalized logistic function with a mean loss and that is often
used in Neural Networks.

\subsubsection{Highly Dimensional Quadratic Function}

In \citep{BJW}, they consider the quadratic map on \(\bbR^n\)
\begin{equation}
  \label{eq:objfun:quad}
  f(x) = \tfrac12\left\langle x,\Sigma^{-1}x  \right\rangle \,,
\end{equation}
where \(\Sigma\) is the matrix whose elements are \(\Sigma_{ij}=\rho^{|i-j|}\), with \(\rho=0.9\) and \(n=50\), and whose
inverse \(\Sigma^{-1}\) is the tridiagonal matrix
\[
  \Sigma^{-1} = \frac1{1-\rho^2}
  \begin{pmatrix}
      1   &  -\rho\\
    -\rho & 1+\rho^2 & -\rho\\
          &   \ddots & \ddots & \ddots\\
          &          & -\rho  & 1+\rho^2 & -\rho\\
          &          &        &  -\rho   &   1
  \end{pmatrix} \,.
\]

\subsubsection{Generalized Rosenbrock Function}

The Rosenbrock function \citep{Ro60}, whose expression is
\[ f(x,y) = (a-x)^2 + b(y-x^2)^2 \,, \]
with \(a,b>0\), represents a banana-shaped flat-valley surrounded by steep walls
with a unique critical point and global minimum at \(a,a^2\), whose search by
numerical means is difficult, hence its use to test and benchmark optimizers.
We consider here its generalization to higher dimensions, \(n>2\), namely
\begin{equation}
  \label{eq:objfun:rosenbrock}
  f(x) = \sum_{i=1}^{n-1}\left[ (1-x_i)^2+100(x_{i+1}-x_i^2)^2 \right] \,.
\end{equation}
As for the two-dimensional case, the function counts with a global minimum at
\((1,1,\dots,1)\) but, unlike it, also has a local minimum close to
\((-1,1,\dots,1)\) (closer the higher is the dimension).

\subsubsection{Yet Another Test Function (YATF)}

Another example that might difficult the search of a minimum is the following
\begin{equation}
  \label{eq:objfun:yatf}
  f(x,y) = \sin(2x^2-y^2+3)\cdot\cos(x+1-\exp(2y)) \,,
\end{equation}
which has a local minimum close to \((0.32,1.60)\).

\subsubsection{Multinomial Logistic Regression}

In artificial neural networks (ANN), the activation function of a node (or neuron) defines the output of that node given an input (or set of inputs). Supervised learning is a learning paradigm of the training process of the ANN. Different choices are available for the activation function and the training process, which is a subject that might be in some cases controversial within the ANN community, but that is not the object of this work.
We consider a shallow neural network with a single layer for classification in \(n\) classes or targets given \(m\) features.
Upon reaching the neurons, the inputs (features) are weighted and possibly biased, what in fact is the model to be determined through the learning process.
\[
w\in\bbR^{m\times n},\ b\in\bbR^n\colon\
x\in\bbR^m \mapsto w\cdot x + b \in\bbR^n
\]
The chosen activation function is a classifier, a generalization of the logistic function, the multinomial logistic function (a.k.a. softmax),
\[ \sigma(z) = \left(\frac{e^{z_j}}{\sum_k e^{z_k}}\right) \,. \]
As loss function, we choose the cross-entropy or log-loss,
\[ \logloss(\hat y, y) = -\left\langle y, \log \hat y\right\rangle \,, \]
where \(\hat y\in\bbR^n\) is the computed output, \(y\in\{0,1\}^n\) with \(\sum_j y_j=1\) is the expected output (class), and \(\log\) is applied componentwise.
We take as objective function the average loss over a training dataset \(\calD\) of length \(|\calD|\), namely,
\begin{equation}
  \label{eq:objfun:logloss}
  f(w,b) = -\frac1{|\calD|}\sum_{(x,y)\in\calD}\left\langle y,(\log\circ\,\sigma)(w\cdot x+b)\right\rangle \,.
\end{equation}

It is important to note here that \(f\) is the sum of convex functions and therefore itself convex, however \(f\) need not have a global minimum but an asymptotic infimum: \(0\) as, for instance, \(\exp(t)\) opposed to \(\cosh(t)\). This depends on the data fed for the training process.

\subsection{Experiments}
\label{sec:experiments}

Many experiments have been performed, we optimized each test function using the aforementioned methods (and others) set with different parameters and initial guesses, from which we present here a small but suggestive sample. In summary, we minimize the test functions \eqref{eq:objfun:rosenbrock} and \eqref{eq:objfun:yatf}, the quadratic function \eqref{eq:objfun:quad}, and the loss \eqref{eq:objfun:logloss} with the PHB/CM and NAG methods \eqref{eq:cmnag} given by the constant, bounded and unbounded coefficients \eqref{eq:coeffs:constant},  \eqref{eq:coeffs:bounded}, and \eqref{eq:coeffs:unbounded}, respectively, and WWJ's method \eqref{eq:WiWiJo16}, the latter three for \(n=3,4\), for a total of seven methods. Each objective is minimized using its own initial guess, time-step, and number of epochs, which are fixed for the seven simulations.

We set Rosenbrock's test function with 30 dimensions and seek for its global minimum at \((1,\dots,1)\) from \((0,\dots,0)\) at a pace of \(h=0.01\) for \(20000\) epochs.
In the case of the YATF, there is a local minimum near \((0.32,1.60)\) which we seek from \((-0.25,0.35)\) with time-step \(h=0.01\) for \(3800\) epochs.
A random point on a sphere of radius \(50\) is the initial guess for the 50-dimensional quadratic function, whose global minimum, sought for \(10000\) epochs with \(h=0.1\), is clearly at the origin.
For the convergence tests, the log-loss function is fed with \(10\) arbitrary samples, with \(4\) features and \(3\) targets each, what defines an optimization problem of dimension \(10\) with no global minimum, hence we will seek for the infimum for \(12000\) epochs at a pace of \(h\approx0.945\) from a random weight distribution with null biases. For an actual ANN test, the log-loss function is fed with the widely used Iris dataset \citep[; see below for more details]{UCIMLR}.

The methods have been implemented in Julia v1.8.2 \citep{bezanson2017julia}, using solely as nonnative libraries NLsolve.jl \citep{mogensen2018nlsolve} to solve the side problem \eqref{eq:argminzp}, Plots.jl \citep{breloff2021plots} and PGFPlotsX.jl for plotting, and Pluto.jl for an interactive notebook. Methods, functions and simulations are available online \citep{cmcampos.xyz,github-cmcampos-xyz}. All plots but Figs. \ref{fig:comparison_cmnag_yatf_trajectory} and \ref{fig:comparison_methods:iris} represent the objective function residual against the epoch in log-log scale.

In Fig. \ref{fig:comparison_cmnag_yatf_trajectory}, we can see how the trajectories
computed by PHB/CM and NAG for the YATF pass by its minimum about \((0.32,1.60)\).
They start at the bottom left and go upward until they ``realize'' they have
overreached the minimum and, about \((0.7,1.9)\), they back up. Although not shown
in the figure, this motion is repeated successively, but each time they back up,
they do so sooner and closer to the minimum, like a heavy ball in a bowl. The trajectories
shown in the figure correspond to the iterations 765 to 1450 of the methods.

\begin{figure}
  \centering
  Yet Another Test Function
  \includegraphics[width=0.8\textwidth]{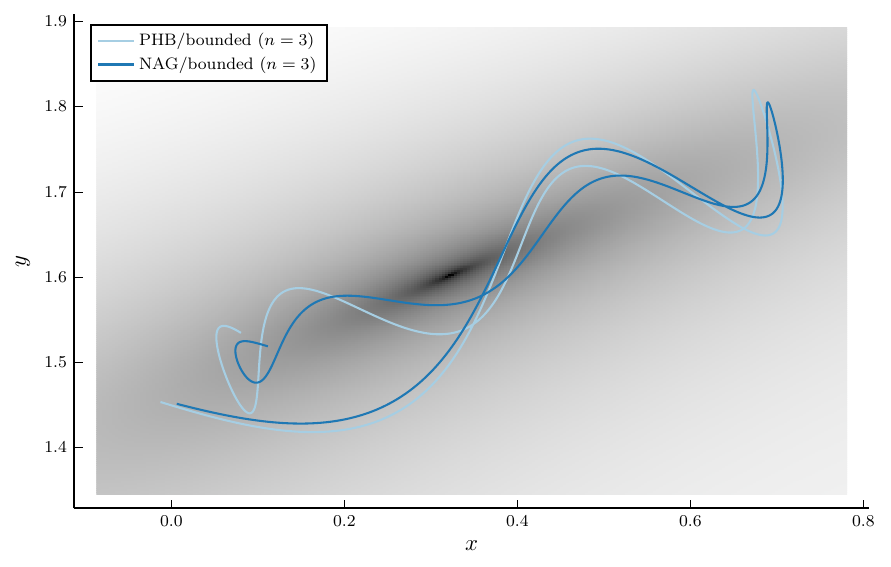}
  \caption{Trajectory slices nearby the local minimum of the YATF using PHB/CM (pale) and NAG (strong) with the bounded coefficients from the Lagrangian's polynomial dilation with \(n=3\). A nonlinear grayscale gradient indicates the minimum's location in black.}
  \label{fig:comparison_cmnag_yatf_trajectory}
\end{figure}

Similarly we observe in Fig. \ref{fig:comparison_cmnag_yatf} that NAG oscillates less than PHB/CM. Each downward peak corresponds to the trajectory passing by the minimum of the YATF. In fact, the trajectories shown in Fig. \ref{fig:comparison_cmnag_yatf_trajectory} correspond to the first peak in blue of Fig. \ref{fig:comparison_cmnag_yatf}.
These oscillations where the trajectories pass by the minimum of the objective function back and forth, and the fact that NAG oscillates less than PHB/CM slightly outperforming it are common aspects of all the simulations performed, reason why in the remaining figures we focus solely on NAG methods, where several aspects are worth to note.

\begin{figure}
  \centering
  Yet Another Test Function
  \includegraphics[width=0.8\textwidth]{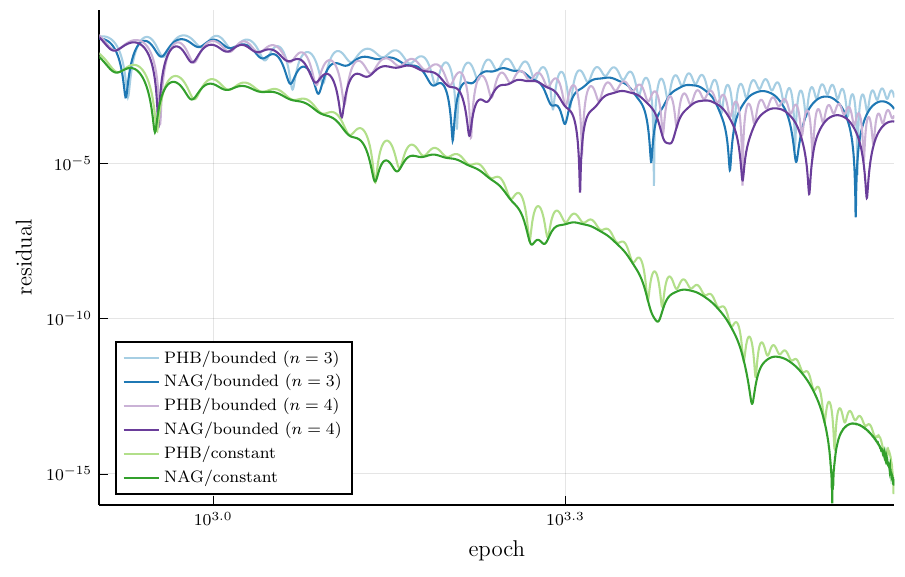}
  \caption{YATF residual values along PHB/CM (pale) and NAG (strong) trajectories for coefficients from the polynomially dilated Lagrangian with \(n=3\) (blue) and \(n=4\) (violet), and from the exponentially dilated Lagrangian (green).}
  \label{fig:comparison_cmnag_yatf}
\end{figure}

Figs. \ref{fig:comparison_methods:hdquad}-\ref{fig:comparison_methods:logloss} compare the seven methods enumerated above, one figure for each objective function presented in section \ref{sec:simulations:objective_functions}, and following the same order. Each figure is made up of two plots: the top one is composed of methods with \(n=3\), the bottom one of methods with \(n=4\), whereas both include the NAG method with constant coefficients for proper reference.

\newcommand{\mycaption}[1]{\caption{#1 values along trajectories computed with NAG for constant (green), bounded (blue), and unbounded (violet) coefficients, and
WWJ 
(red); the latter three set with \(n=3\) (top) and \(n=4\) (bottom).}}
\begin{figure}
  \centering
  Quadratic function (dimension 50)
  \includegraphics[width=1.0\textwidth]{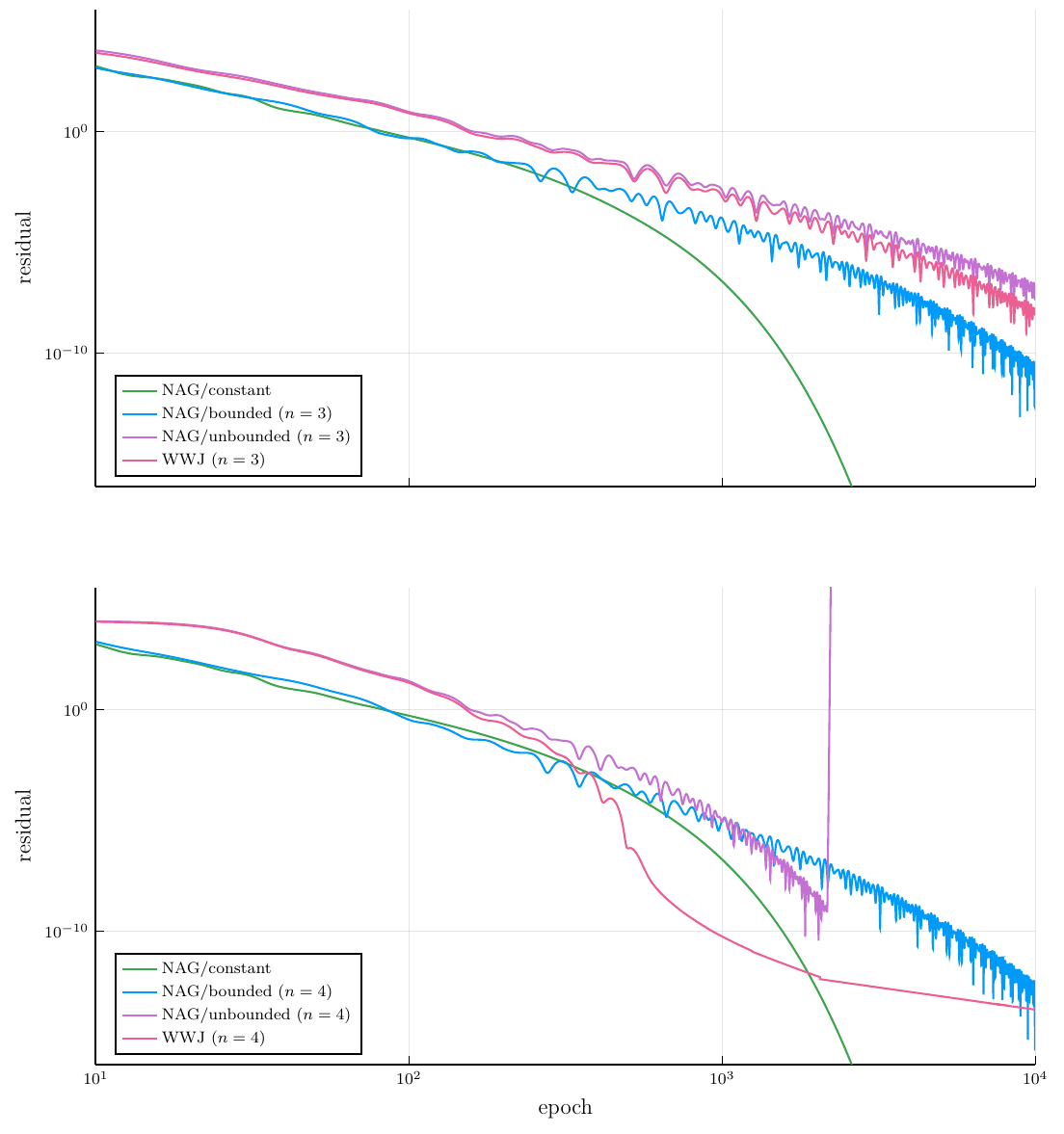}
  \mycaption{Quadratic test function}
  \label{fig:comparison_methods:hdquad}
\end{figure}
\begin{figure}
  \centering
  Rosenbrock's test function (dimension 30)
  \includegraphics[width=1.0\textwidth]{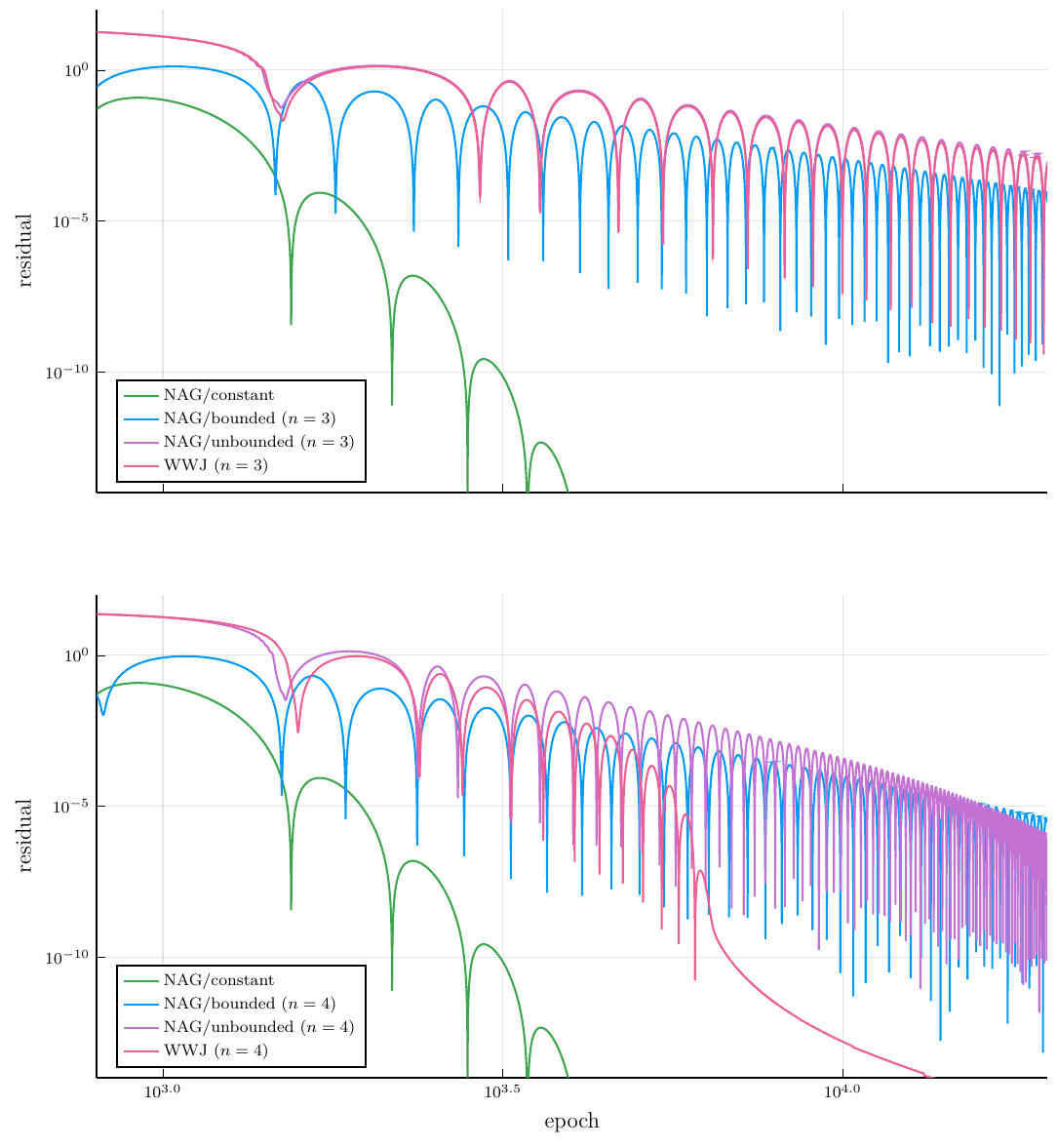}
  \mycaption{Rosenbrock's test function}
  \label{fig:comparison_methods:rosenbrock}
\end{figure}
\begin{figure}
  \centering
  Yet Another Test Function (dimension 2)
  \includegraphics[width=1.0\textwidth]{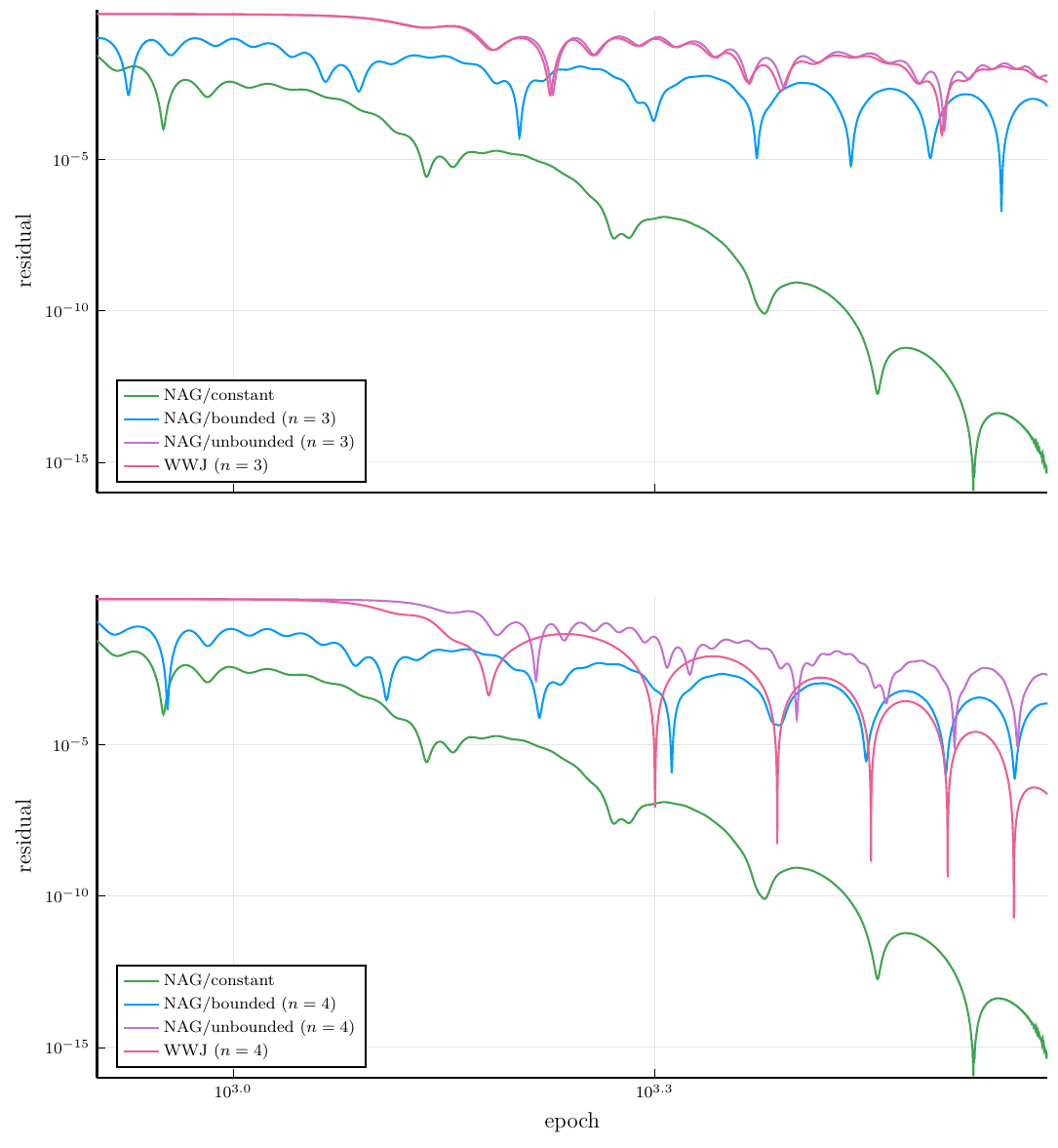}
  \mycaption{YATF residual}
  \label{fig:comparison_methods:yatf}
\end{figure}
\begin{figure}
  \centering
  Multinomial logistic loss (dimension 10)
  \includegraphics[width=1.0\textwidth]{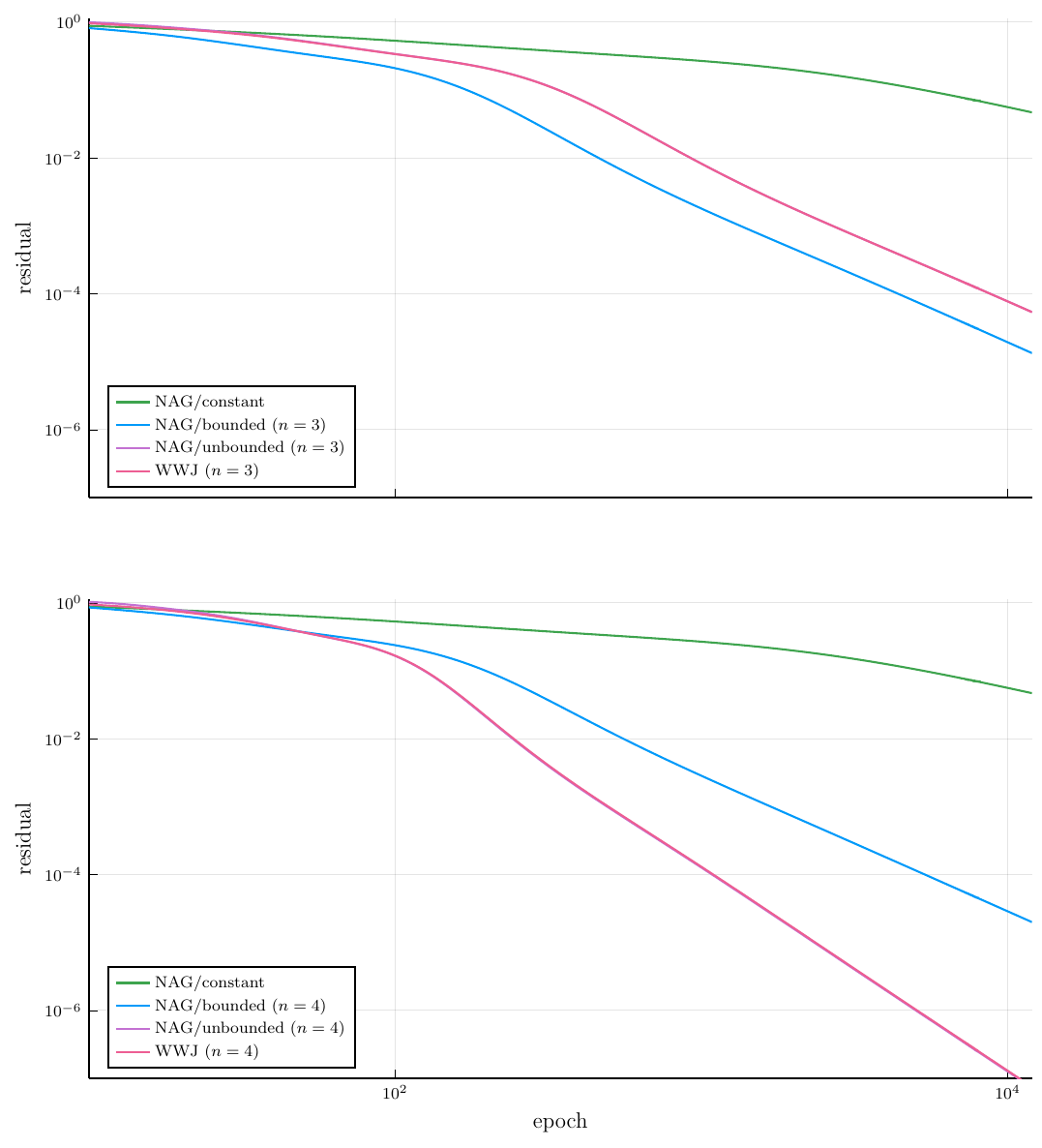}
  \mycaption{Logistic loss}
  \label{fig:comparison_methods:logloss}
\end{figure}

Observe first that, the three methods with \(n=3\) behave similarly for the four objective functions. Not surprising since, the Lagrangians considered reduce to a potentially dilated mechanical Lagrangian in which the objective function, in the cases of unbounded NAG and WWJ, is coupled with a further constant coefficient, \(D=\sfrac14\) when \(n=3\), which slightly delays convergence when compared to bounded NAG. This coefficient needs to be small in order to ensure convergence according to \cite{WiWiJo16}.

When \(n=4\), these three methods clearly show their differences. Most importantly, unbounded NAG \eqref{eq:coeffs:unbounded} blows up in Fig. \ref{fig:comparison_methods:hdquad} (highly dimensional quadratic function). This is due to the increasing and unbounded learning rate coefficient \(\eta\) that is ``ignored'' when the gradient is almost null but brakes the method when it steps at a point where the gradient is not negligible. Therefore we infer that unbounded NAG is not suited for long runs when \(n\geq4\) and should be restarted to avoid blow-up. Surprisingly WWJ is not affected by this problem even though the same coefficient appears in its definition, Eq. \eqref{eq:WiWiJo16}. The method is numerically stable thanks to the fast convergence of the trajectory towards the minimum.

In fact, WWJ really shows up when \(n=4\), where it clearly outperforms its NAG counterpart, unbounded NAG, as well as bounded NAG, but may ``stall'' when the simulation has advanced, Fig. \ref{fig:comparison_methods:hdquad}. Besides there is a trick into its performance, when \(n\geq4\) the method must solve a side optimization problem, \eqref{eq:argminzp}, which is solved here using the NLsolve.jl library, something that it somewhat redundant and suffers from the curse of dimensionality: it is around 10 times slower than the other methods for the low dimensional case (YATF) and more than 100 times slower in the high dimensional ones (Rosenbrock and quadratic function). Nonetheless, this disadvantage might decrease or even disappear when the Bregman divergence is not the simple Euclidean norm (something worth to explore), since the associated methods are unlikely to be explicit.

With regards to constant NAG, the results are somehow paradoxical. Recall that the exponentially dilated Lagrangian is the only case in which a good convergence rate is not ensured, Eq. \eqref{eq:ideal_exponents:exp_dilation}, however it is the method that in general performs the best, Figs. \ref{fig:comparison_methods:hdquad}-\ref{fig:comparison_methods:yatf}. Furthermore, although within the Machine Learning community it is perhaps the most popular method among the analysed, it is the one that has performed the worst in the purely ML scenario, Fig. \ref{fig:comparison_methods:logloss}. 

Fig. \ref{fig:comparison_methods:logloss} is where the theoretical convergence rates are better seen. In this figure, the values obtained by unbounded NAG and WWJ coalesce for both, \(n=3\) and \(n=4\).

In addition to the previous optimization problems, we apply the analyzed methods to a simple classification problem using Firsher's Iris dataset \citep{UCIMLR}. The dataset consists of 150 entries (or samples) with 7 fields: 4 features and 3 targets. Therefore we consider a shallow neural network with a single layer of 3 neurons with 4 inputs. The input data (the features) are weighted and biases upon entry into the network, the computed output is compared with the expected output (targets) using the multinomial logistic function. This process is summarized by the objective function \eqref{eq:objfun:logloss}.

For simplicity, we only consider four methods: constant, bounded, and unbounded NAG, and WWJ, the latter three with \(n=3\). We train the network (or optimize the objective function) at an increasing number of epochs (from 25 up to 250) for 1000 runs. At each run and for each number of epochs, the samples are randomly split in two: 100 samples for training and 50 samples for testing. At the same time, the initial weights of the network are drawn from a normal distribution with null mean and standard deviation \(\sigma=10\). This initial guess is kept for the four methods. Then the network is trained and the accuracy of network with the computed weights is measured (percentage of correct matches for the testing data).

Fig. \ref{fig:comparison_methods:iris} shows the average accuracy along the 1000 runs for each method \emph{versus} the number of epochs. This figure is clearly consistent with what is obtained in Fig. \ref{fig:comparison_methods:logloss}.

\begin{figure}
  \centering
  Multinomial logistic loss
  \includegraphics[width=0.8\textwidth]{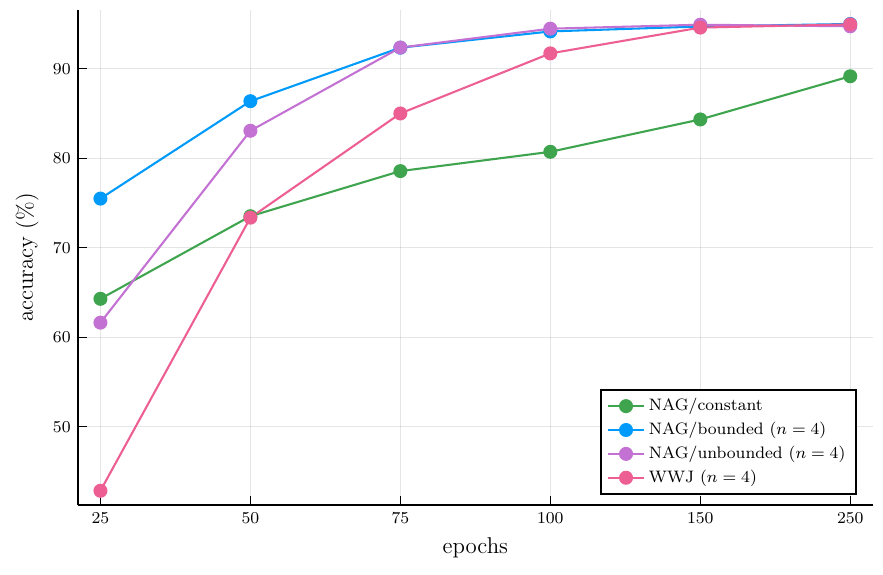}
  \caption{Average accuracies for Iris dataset achieved by NAG with constant (green), bounded (blue), and unbounded (violet) coefficients, and WWJ (red); the latter three set with \(n=4\).}
  \label{fig:comparison_methods:iris}
\end{figure}

\section{Conclusions and Future work}
In this paper, we have studied the relation between accelerated optimization and discrete variational calculus, proving a symplecticity property for the continuous differential equation in Theorem \ref{thm:wer} which is also preserved by the corresponding discrete variational methods. We have derived Classical Momentum (CM) or Polyak's Heavy Ball (PHB), and Nesterov’s Accelerated Gradient methods (NAG) from the discrete Hamiltonian and Lagrange-d'Alembert principles in Theorem \ref{thm:cm-vs-nag} adding forces in the picture and proven a one-to-one correspondence. Several simulations were performed showing the applicability of our techniques to optimization. Among all the methods, NAG with the constant coefficients from the exponentially dilated Lagrangian, aside of being the simplest choice, it also seems to be the best one for general purpose applications according to the simulations.

In a future paper, we will study that the proposed  optimization algorithm generated by using Lagrange-d'Alembert principle achieves the accelerated rate for minimizing both strongly convex functions and convex functions \citep{WiWiJo16,Shi}.
The main idea is to discretize, using discrete variational calculus, the  continuous Euler-Lagrange equations (with or without forces)  while maintaining
their convergence rates \citep[see][, for recent advances in this topic]{vaquero}.
Moreover, the extension to problems of accelerated optimization in manifolds will given using discrete variational calculus and well-know optimization techniques with retraction maps (\citealp{AbMaSeBookRetraction}; see also \citealp{LeeTaoLeok}).


\acks{D. Martín de Diego acknowledges financial support from the Spanish Ministry of Science and Innovation, under grants PID2019-106715GB-C21, from the Spanish National Research Council (CSIC), through the ``Ayuda extraordinaria a Centros de Excelencia Severo Ochoa'' R\&D  (CEX2019-000904-S). A. Mahillo would like to thank CSIC for its financial support through a JAE Intro scholarship.}

\bibliography{references}

\begin{thebibliography}{42}
\providecommand{\natexlab}[1]{#1}
\providecommand{\url}[1]{\texttt{#1}}
\expandafter\ifx\csname urlstyle\endcsname\relax
  \providecommand{\doi}[1]{doi: #1}\else
  \providecommand{\doi}{doi: \begingroup \urlstyle{rm}\Url}\fi

\bibitem[Abraham and Marsden(1978)]{Abraham1978}
Ralph Abraham and Jerrold~E. Marsden.
\newblock \emph{Foundations of Mechanics}.
\newblock {AMS Chelsea Publishing}, {Redwood City, CA}, 2 edition, 1978.

\bibitem[Absil et~al.(2008)Absil, Mahony, and Sepulchre]{AbMaSeBookRetraction}
P.-A. Absil, R.~Mahony, and R.~Sepulchre.
\newblock \emph{Optimization algorithms on matrix manifolds}.
\newblock Princeton University Press, Princeton, NJ, 2008.
\newblock ISBN 978-0-691-13298-3.
\newblock \doi{10.1515/9781400830244}.
\newblock URL \url{https://doi.org/10.1515/9781400830244}.
\newblock With a foreword by Paul Van Dooren.

\bibitem[Attouch et~al.(2021)Attouch, Chbani, and Riahi]{Attouch21}
Hedy Attouch, Zaki Chbani, and Hassan Riahi.
\newblock Fast convex optimization via time scaling of damped inertial gradient
  dynamics.
\newblock \emph{Pure Appl. Funct. Anal.}, 6\penalty0 (6):\penalty0 1081--1117,
  2021.
\newblock ISSN 2189-3756.

\bibitem[Betancourt et~al.(2018)Betancourt, Jordan, and Wilson]{BJW}
Michael Betancourt, Michael~I. Jordan, and Ashia~C. Wilson.
\newblock On symplectic optimization.
\newblock \emph{arXiv}, 1802.03653, 2018.

\bibitem[Bezanson et~al.(2017)Bezanson, Edelman, Karpinski, and
  Shah]{bezanson2017julia}
Jeff Bezanson, Alan Edelman, Stefan Karpinski, and Viral~B Shah.
\newblock Julia: A fresh approach to numerical computing.
\newblock \emph{SIAM review}, 59\penalty0 (1):\penalty0 65--98, 2017.
\newblock URL \url{https://doi.org/10.1137/141000671}.

\bibitem[Blanes and Casas(2016)]{blanes}
Sergio Blanes and Fernando Casas.
\newblock \emph{A concise introduction to geometric numerical integration}.
\newblock Monographs and Research Notes in Mathematics. CRC Press, Boca Raton,
  FL, 2016.
\newblock ISBN 978-1-4822-6342-8.

\bibitem[Br\`egman(1967)]{bregman}
L.~M. Br\`egman.
\newblock A relaxation method of finding a common point of convex sets and its
  application to the solution of problems in convex programming.
\newblock \emph{\v{Z}. Vy\v{c}isl. Mat i Mat. Fiz.}, 7:\penalty0 620--631,
  1967.
\newblock ISSN 0044-4669.

\bibitem[Breloff(2021)]{breloff2021plots}
Tom Breloff.
\newblock Plots.jl, May 2021.
\newblock URL \url{https://doi.org/10.5281/zenodo.4776893}.

\bibitem[Campos(2014)]{Campos14}
C\'{e}dric~M. Campos.
\newblock High order variational integrators: a polynomial approach.
\newblock In \emph{Advances in differential equations and applications},
  volume~4 of \emph{SEMA SIMAI Springer Ser.}, pages 249--258. Springer, Cham,
  2014.
\newblock \doi{10.1007/978-3-319-06953-1\_24}.
\newblock URL \url{https://doi.org/10.1007/978-3-319-06953-1_24}.

\bibitem[Campos and Sanz-Serna(2017)]{CamposSS17}
C\'{e}dric~M. Campos and J.~M. Sanz-Serna.
\newblock Palindromic 3-stage splitting integrators, a roadmap.
\newblock \emph{J. Comput. Phys.}, 346:\penalty0 340--355, 2017.
\newblock ISSN 0021-9991.
\newblock \doi{10.1016/j.jcp.2017.06.006}.
\newblock URL \url{https://doi.org/10.1016/j.jcp.2017.06.006}.

\bibitem[Campos(2022{\natexlab{a}})]{cmcampos.xyz}
Cédric~M. Campos.
\newblock Personal website, 2022{\natexlab{a}}.
\newblock URL \url{https://cmcampos.xyz}.

\bibitem[Campos(2022{\natexlab{b}})]{github-cmcampos-xyz}
Cédric~M. Campos.
\newblock Research repository, 2022{\natexlab{b}}.
\newblock URL \url{https://github.com/cmcampos-xyz}.

\bibitem[Cappelletti-Montano et~al.(2013)Cappelletti-Montano, De~Nicola, and
  Yudin]{Nicola}
Beniamino Cappelletti-Montano, Antonio De~Nicola, and Ivan Yudin.
\newblock A survey on cosymplectic geometry.
\newblock \emph{Rev. Math. Phys.}, 25\penalty0 (10):\penalty0 1343002, 55,
  2013.
\newblock ISSN 0129-055X.
\newblock \doi{10.1142/S0129055X13430022}.
\newblock URL \url{https://doi.org/10.1142/S0129055X13430022}.

\bibitem[Cauchy(1847)]{CauchyGD}
Agustin-Louis Cauchy.
\newblock Méthode générale pour la résolution des systèmes d'équations
  simultanées.
\newblock \emph{C. R. Acad. Sci.}, 25:\penalty0 536–--538, 1847.
\newblock URL \url{https://gallica.bnf.fr/ark:/12148/bpt6k2982c/f540.item}.

\bibitem[Celledoni et~al.(2020)Celledoni, Ehrhardt, Christian, McLachlan,
  Owren, Sch\"{o}nlieb, and Ferdia]{CellEhO}
Elena Celledoni, Matthias~J. Ehrhardt, Etmann Christian, Robert~I McLachlan,
  Brynjulf Owren, Carola-Bibiane Sch\"{o}nlieb, and Sherry Ferdia.
\newblock Structure preserving deep learning.
\newblock \emph{arXiv}, 2006.03364, 2020.
\newblock URL \url{https://arxiv.org/abs/2006.03364}.

\bibitem[{de Le{\'o}n} and R.~Rodrigues(1987)]{deLeon1987}
Manuel {de Le{\'o}n} and Paulo R.~Rodrigues.
\newblock \emph{Methods of Differential Geometry in Analytical Mechanics},
  volume 158.
\newblock {Elsevier}, {Amsterdam}, 1987.
\newblock ISBN 0-08-087269-7.

\bibitem[Dua and Graff(2017)]{UCIMLR}
Dheeru Dua and Casey Graff.
\newblock {UCI} machine learning repository, 2017.
\newblock URL \url{http://archive.ics.uci.edu/ml}.

\bibitem[Duruisseaux et~al.(2021)Duruisseaux, Schmitt, and Leok]{DSLeok}
Valentin Duruisseaux, Jeremy Schmitt, and Melvin Leok.
\newblock Adaptive {H}amiltonian variational integrators and applications to
  symplectic accelerated optimization.
\newblock \emph{SIAM J. Sci. Comput.}, 43\penalty0 (4):\penalty0 A2949--A2980,
  2021.
\newblock ISSN 1064-8275.
\newblock \doi{10.1137/20M1383835}.
\newblock URL \url{https://doi.org/10.1137/20M1383835}.

\bibitem[Hairer et~al.(2010)Hairer, Lubich, and Wanner]{hairer}
E.~Hairer, C.~Lubich, and G.~Wanner.
\newblock \emph{Geometric numerical integration}, volume~31 of \emph{Springer
  Series in Computational Mathematics}.
\newblock Springer, Heidelberg, 2010.
\newblock ISBN 978-3-642-05157-9.
\newblock Structure-preserving algorithms for ordinary differential equations,
  Reprint of the second (2006) edition.

\bibitem[Hartman(2002)]{hartman}
P.~Hartman.
\newblock \emph{Ordinary differential equations}, volume~38 of \emph{Classics
  in Applied Mathematics}.
\newblock Society for Industrial and Applied Mathematics (SIAM), Philadelphia,
  PA, 2002.
\newblock ISBN 0-89871-510-5.
\newblock \doi{10.1137/1.9780898719222}.
\newblock URL \url{http://dx.doi.org/10.1137/1.9780898719222}.
\newblock Corrected reprint of the second (1982) edition [Birkh{\"a}user,
  Boston, MA; MR0658490 (83e:34002)], With a foreword by Peter Bates.

\bibitem[Jordan(2018)]{Jo18}
Michael~I. Jordan.
\newblock Dynamical symplectic and stochastic perspectives on gradient-based
  optimization.
\newblock In \emph{Proceedings of the {I}nternational {C}ongress of
  {M}athematicians---{R}io de {J}aneiro 2018. {V}ol. {I}. {P}lenary lectures},
  pages 523--549. World Sci. Publ., Hackensack, NJ, 2018.

\bibitem[Lee et~al.(2021)Lee, Tao, and Leok]{LeeTaoLeok}
Taeyoung Lee, Molei Tao, and Melvin Leok.
\newblock Variational symplectic accelerated optimization on lie groups.
\newblock \emph{arXiv}, 2103.14166, 2021.
\newblock URL \url{https://arxiv.org/abs/2103.14166}.

\bibitem[Libermann(1959)]{Libermann}
Paulette Libermann.
\newblock Sur les automorphismes infinit\'{e}simaux des structures
  symplectiques et des structures de contact.
\newblock In \emph{Colloque {G}\'{e}om. {D}iff. {G}lobale ({B}ruxelles, 1958)},
  pages 37--59. Centre Belge Rech. Math., Louvain, 1959.

\bibitem[Libermann and Marle(1987)]{marle}
Paulette Libermann and Charles-Michel Marle.
\newblock \emph{Symplectic geometry and analytical mechanics}, volume~35 of
  \emph{Mathematics and its Applications}.
\newblock D. Reidel Publishing Co., Dordrecht, 1987.
\newblock ISBN 90-277-2438-5.
\newblock \doi{10.1007/978-94-009-3807-6}.
\newblock URL \url{https://doi.org/10.1007/978-94-009-3807-6}.
\newblock Translated from the French by Bertram Eugene Schwarzbach.

\bibitem[Marrero et~al.(2016)Marrero, de~Diego, and Martínez]{MMdDM2016}
J.~C. Marrero, D.~Martín de~Diego, and E.~Martínez.
\newblock On the exact discrete lagrangian function for variational
  integrators: theory and applications, 2016.
\newblock URL \url{https://arxiv.org/abs/1608.01586}.

\bibitem[Marsden and West(2001)]{marsden-west}
J.~E. Marsden and M.~West.
\newblock Discrete mechanics and variational integrators.
\newblock \emph{Acta Numer.}, 10:\penalty0 357--514, 2001.
\newblock ISSN 0962-4929.
\newblock \doi{10.1017/S096249290100006X}.
\newblock URL \url{http://dx.doi.org/10.1017/S096249290100006X}.

\bibitem[Marthinsen and Owren(2016)]{MaOw16}
H{\aa}kon Marthinsen and Brynjulf Owren.
\newblock Geometric integration of non-autonomous linear {H}amiltonian
  problems.
\newblock \emph{Adv. Comput. Math.}, 42\penalty0 (2):\penalty0 313--332, 2016.
\newblock ISSN 1019-7168.
\newblock \doi{10.1007/s10444-015-9425-0}.
\newblock URL \url{https://doi.org/10.1007/s10444-015-9425-0}.

\bibitem[Mogensen et~al.(2020)Mogensen, Carlsson, Villemot, Lyon, Gomez,
  Rackauckas, Holy, Widmann, Kelman, Karrasch, Levitt, Riseth, Lucibello, Kwon,
  Barton, TagBot, Baran, Lubin, Choudhury, Byrne, Christ, Arakaki, Bojesen,
  benneti, and Macedo]{mogensen2018nlsolve}
Patrick~Kofod Mogensen, Kristoffer Carlsson, Sébastien Villemot, Spencer Lyon,
  Matthieu Gomez, Christopher Rackauckas, Tim Holy, David Widmann, Tony Kelman,
  Daniel Karrasch, Antoine Levitt, Asbjørn~Nilsen Riseth, Carlo Lucibello,
  Changhyun Kwon, David Barton, Julia TagBot, Mateusz Baran, Miles Lubin,
  Sarthak Choudhury, Simon Byrne, Simon Christ, Takafumi Arakaki,
  Troels~Arnfred Bojesen, benneti, and Miguel Raz~Guzmán Macedo.
\newblock Julianlsolvers/nlsolve.jl: v4.5.1, December 2020.
\newblock URL \url{https://doi.org/10.5281/zenodo.4404703}.

\bibitem[Nesterov(1983)]{Nesterov}
Yu.~E. Nesterov.
\newblock A method for solving the convex programming problem with convergence
  rate {$O(1/k^{2})$}.
\newblock \emph{Dokl. Akad. Nauk SSSR}, 269\penalty0 (3):\penalty0 543--547,
  1983.
\newblock ISSN 0002-3264.

\bibitem[Nesterov(2018)]{Book-Nesterov}
Yurii Nesterov.
\newblock \emph{Lectures on convex optimization}, volume 137 of \emph{Springer
  Optimization and Its Applications}.
\newblock Springer, Cham, 2018.
\newblock ISBN 978-3-319-91577-7; 978-3-319-91578-4.
\newblock \doi{10.1007/978-3-319-91578-4}.
\newblock URL \url{https://doi.org/10.1007/978-3-319-91578-4}.
\newblock Second edition of [ MR2142598].

\bibitem[Patrick and Cuell(2009)]{PatrickCuell}
G.~W. Patrick and C.~Cuell.
\newblock Error analysis of variational integrators of unconstrained
  {L}agrangian systems.
\newblock \emph{Numer. Math.}, 113\penalty0 (2):\penalty0 243--264, 2009.
\newblock ISSN 0029-599X.
\newblock \doi{10.1007/s00211-009-0245-3}.
\newblock URL \url{http://dx.doi.org/10.1007/s00211-009-0245-3}.

\bibitem[Polak(1997)]{Polak-book}
Elijah Polak.
\newblock \emph{Optimization}, volume 124 of \emph{Applied Mathematical
  Sciences}.
\newblock Springer-Verlag, New York, 1997.
\newblock ISBN 0-387-94971-2.
\newblock \doi{10.1007/978-1-4612-0663-7}.
\newblock URL \url{https://doi.org/10.1007/978-1-4612-0663-7}.
\newblock Algorithms and consistent approximations.

\bibitem[Polyak(1964)]{Po64}
Boris~T. Polyak.
\newblock Some methods of speeding up the convergence of iterative methods.
\newblock \emph{\v{Z}. Vy\v{c}isl. Mat i Mat. Fiz.}, 4:\penalty0 791--803,
  1964.
\newblock ISSN 0044-4669.

\bibitem[Polyak(1987)]{Po87}
Boris~T. Polyak.
\newblock \emph{Introduction to optimization}.
\newblock Translations Series in Mathematics and Engineering. Optimization
  Software, Inc., Publications Division, New York, 1987.
\newblock ISBN 0-911575-14-6.
\newblock Translated from the Russian, With a foreword by Dimitri P. Bertsekas.

\bibitem[Rosenbrock(1960/61)]{Ro60}
H.~H. Rosenbrock.
\newblock An automatic method for finding the greatest or least value of a
  function.
\newblock \emph{Comput. J.}, 3:\penalty0 175--184, 1960/61.
\newblock ISSN 0010-4620.
\newblock \doi{10.1093/comjnl/3.3.175}.
\newblock URL \url{https://doi.org/10.1093/comjnl/3.3.175}.

\bibitem[Sanz-Serna and Calvo(1994)]{serna}
J.~M. Sanz-Serna and M.~P. Calvo.
\newblock \emph{Numerical {H}amiltonian problems}, volume~7 of \emph{Applied
  Mathematics and Mathematical Computation}.
\newblock Chapman \& Hall, London, 1994.
\newblock ISBN 0-412-54290-0.

\bibitem[Shi et~al.(2019)Shi, S., M.I., and J.U.]{Shi}
B~Shi, Du~S. S., Jordan M.I., and Su~J.U.
\newblock Acceleration via symplectic discretization of high-resolution
  differential equations.
\newblock \emph{arXiv}, 2019.
\newblock URL \url{https://arxiv.org/pdf/1902.03694.pdf}.

\bibitem[Su et~al.(2016)Su, Boyd, and Cand{{\`e}}s]{SuBoCa16}
Weijie Su, Stephen Boyd, and Emmanuel~J. Cand{{\`e}}s.
\newblock A differential equation for modeling nesterov's accelerated gradient
  method: Theory and insights.
\newblock \emph{Journal of Machine Learning Research}, 17\penalty0
  (153):\penalty0 1--43, 2016.
\newblock URL \url{http://jmlr.org/papers/v17/15-084.html}.

\bibitem[Sutskever et~al.(2013)Sutskever, Martens, Dahl, and
  Hinton]{pmlr-v28-sutskever13}
Ilya Sutskever, James Martens, George Dahl, and Geoffrey Hinton.
\newblock On the importance of initialization and momentum in deep learning.
\newblock In Sanjoy Dasgupta and David McAllester, editors, \emph{Advances in
  Neural Information Processing Systems}, volume~28 of \emph{Proceedings of
  Machine Learning Research}, pages 1139--1147, Atlanta, Georgia, USA, 17--19
  Jun 2013. PMLR.
\newblock URL \url{http://proceedings.mlr.press/v28/sutskever13.html}.

\bibitem[Vaquero et~al.(2021)Vaquero, P., and Cort\'es]{vaquero}
M.~Vaquero, Mestres P., and J:~Cort\'es.
\newblock Resource-aware discretization of accelerated optimization flows.
\newblock \emph{Preprint}, 2021.
\newblock URL
  \url{http://carmenere.ucsd.edu/jorge/publications/data/2020_VaMeCo-tac.pdf}.

\bibitem[Wibisono et~al.(2016)Wibisono, Wilson, and Jordan]{WiWiJo16}
Andre Wibisono, Ashia~C. Wilson, and Michael~I. Jordan.
\newblock A variational perspective on accelerated methods in optimization.
\newblock \emph{Proc. Natl. Acad. Sci. USA}, 113\penalty0 (47):\penalty0
  E7351--E7358, 2016.
\newblock ISSN 0027-8424.
\newblock \doi{10.1073/pnas.1614734113}.
\newblock URL \url{https://doi.org/10.1073/pnas.1614734113}.

\bibitem[Wibisono(2016)]{Wi16}
Andre~Yohannes Wibisono.
\newblock \emph{Variational and {D}ynamical {P}erspectives {O}n {L}earning and
  {O}ptimization}.
\newblock ProQuest LLC, Ann Arbor, MI, 2016.
\newblock ISBN 978-1369-05764-5.
\newblock Thesis (Ph.D.)--University of California, Berkeley.

\end{thebibliography}
\end{document}